\documentclass[12pt]{amsproc}

\usepackage{fullpage}
\usepackage[utf8]{inputenc}
\usepackage[T1]{fontenc}
\usepackage{graphicx,subfigure}
\usepackage{amsmath,amsfonts,amssymb,mathtools}
\usepackage{stmaryrd}
\usepackage{mathrsfs,dsfont}
\usepackage{amsthm}
\usepackage{mathabx}
\usepackage{tabularx}
\usepackage{float}

\usepackage{hyperref}

\usepackage{enumerate}

\usepackage{color}

\newcommand{\E}{\mathbb{E}}
\newcommand{\R}{\mathbb{R}}
\newcommand{\N}{\mathbb{N}}
\newcommand{\G}{\mathcal{G}}
\newcommand{\F}{\mathcal{F}}
\newcommand{\RR}{\mathcal{R}}
\newcommand{\al}{\overline{\alpha}}

\newtheorem{theo}{Theorem}[section]

\newtheorem{rem}[theo]{Remark}

\newtheorem{propo}[theo]{Proposition}
\newtheorem{lemma}[theo]{Lemma}

\newtheorem{as}{Assumption}

\allowdisplaybreaks[4]

\begin{document}

\title{On parareal algorithms for semilinear parabolic Stochastic PDEs}

\author{Charles-Edouard~Br\'ehier}
\address{Univ Lyon, CNRS, Universit\'e Claude Bernard Lyon 1, UMR5208, Institut Camille Jordan, F-69622 Villeurbanne, France}
\email{brehier@math.univ-lyon1.fr}

\author{Xu Wang}
\address{
Department of Mathematics, 
Purdue University, 
150 North University Street, 
West Lafayette, Indiana 47907, USA
}
\email{wang4191@purdue.edu}

\maketitle

\begin{abstract}
Parareal algorithms are studied for semilinear parabolic stochastic partial differential equations. These algorithms proceed as two-level integrators, with fine and coarse schemes, and have been designed to achieve a ``parallel in real time'' implementation. In this work, the fine integrator is given by the exponential Euler scheme. Two choices for the coarse integrator are considered: the linear implicit Euler scheme, and the exponential Euler scheme.

The influence on the performance of the parareal algorithm, of the choice of the coarse integrator, of the regularity of the noise, and of the number of parareal iterations, is investigated, with theoretical analysis results and with extensive numerical experiments.
\end{abstract}

\section{Introduction}

In the last two decades, numerical methods for Stochastic Partial Differential Equations (SPDEs) have been extensively studied, see for instance the monographs~\cite{Jentzen_Kloeden:11},~\cite{Kruse:14},~\cite{Lord_Powell_Shardlow:14} and references therein. The rate of convergence of the schemes used for temporal and spatial discretization is related to the regularity of the noise, which may be arbitrarily low, and in such situation effective numerical methods are difficult to construct. For instance, for a one-dimensional semilinear parabolic SPDE, driven by Gaussian space-time white noise, trajectories are only $\alpha$-H\"older continuous in time and $2\alpha$-H\"older continuous in space, for $\alpha<\frac14$, and standard Euler and finite difference schemes thus have a (strong) rate of convergence equal to $\alpha$ and $2\alpha$ respectively.

In this article, we will only focus on the temporal discretization and consider a semilinear parabolic SPDE of the type
\begin{equation*}
\begin{aligned}
\frac{\partial u(t,x)}{\partial t}=&~\frac{\partial^2u(t,x)}{\partial x^2}+F(u(t,x))+\dot{W}^Q(t,x),\quad(t,x)\in\R_+\times(0,1),\\
u(0,x)=&~u_0(x),\quad x\in(0,1),\quad u(t,0)=u(t,1)=0,
\end{aligned}
\end{equation*}
that is, a one-dimensional semilinear heat equation with homogeneous Dirichlet boundary conditions, and with additive Gaussian noise, which is white in time and colored in space. A rigorous interpretation as a stochastic evolution equation (in the sense of~\cite{DaPrato_Zabczyk:14}), driven by a Wiener process, with values in an infinite dimensional Hilbert space is given by~\eqref{eq:SPDE}, see also Section~\ref{sec:2}, where precise assumptions for the nonlinear operator $f$ and the covariance operator $Q$ are stated. It is well-known that the rate of convergence (studied only in the strong sense in this article) of the error depends on the properties of the covariance operator $Q$ of the Wiener process, and can be arbitrarily small. In particular, in the case of space-time white noise ($Q$ is the identity), the order of convergence of Euler schemes is essentially $\frac14$. If the noise is of trace-class ($Q$ has finite trace), the order of convergence is essentially $\frac12$.

The objective of this article is to study the applicability of so-called parareal algorithms to improve the performance of temporal discretization schemes. Implementing parallel algorithms to solve time-dependent evolution equations is not natural, and parallel-in-time integration methods have been extensively studied, based on multigird or multiple shooting techniques. We refer to the pionneering contributions~\cite{BZ89,CP93,HV95,K94,LO87,ML67,N64}, and the review~\cite{G}. The parareal algorithm has been introduced in~\cite{LMT01}, see also~\cite{MT}. It is a {\it parallel in real time} technique, based on the use of two integrators with two different time-step sizes: a coarse integrator, denoted by $\G$, with coarse time-step size $\Delta T$, and a fine integrator, denoted by $\F$, using $J$ steps with time-step size $\delta t$, such that $\Delta T=J\delta t$. In practice, $\G$ may be less accurate but cheaper than $\F$. The parareal algorithm is an iterative method, using a predictor-corrector strategy, in which computations of the fine integrator at each iteration are performed in parallel: the recursion is given by
\begin{equation*}
\begin{aligned}
u_{n+1}^{(k+1)}&=\G(u_{n}^{(k+1)},t_n,t_{n+1})+\F(u_{n}^{(k)},t_n,t_{n+1})-\G(u_{n}^{(k)},t_n,t_{n+1}),\\
u_{0}^{(k)}&=u_0,
\end{aligned}
\end{equation*}
where $k$ represents the index for parareal iterations. For $k=0$, the scheme is initialized using the coarse integrator. We refer to Equation~\eqref{eq:parareal} and to Section~\ref{sec:3} below for a detailed presentation of the algorithm.

Since the pioneering work~\cite{LMT01}, where the parareal algorithm has been introduced for a class of ordinary differential equations, several extensions have been considered, see for instance~\cite{B05,GH08,GH14,GV07,GV07-2,MT02}. In particular,~\cite{B05} deals with parabolic PDEs, and studies the stability and convergence properties, which may require regularity properties, depending on the choice of integrators. The application of the parareal algorithm for stochastic systems has been considered first in~\cite{B06}, and more recently in~\cite{HWZ} for stochastic Schr\"odinger PDEs and in~\cite{ZZJ} for a class of stochastic differential equations. More precisely, in~\cite{HWZ}, parareal algorithms for stochastic Schr\"odinger equation with damping are studied with $\F$ being the exact solver and $\G$ being the exponential-$\theta$ scheme. The longterm convergence is obtained for the case $\theta>\frac12$ or sufficient large $\alpha$, which ensures sufficient exponential decay of the coarse integrator.

Let us now describe the contributions of this article. The parareal algorithm is applied to the SPDE above. The exponential Euler scheme is chosen as the fine integrator. The main contribution of this article is to reveal that the parareal algorithm behaves differently, depending on the choice of the coarse integrator, when applied to semilinear parabolic SPDEs. Such results, based on both theoretical analysis and numerical experiments, have not been reported before, up to our knowledge. The error of the parareal algorithm, considered in this article, is the distance between the solution computed by the proposed algorithm and a reference solution generated by the fine integrator (which is not computationally expensive and is not computed in practice). The parareal algorithm is useful to reduce computational cost only if the order of convergence of this error (with respect to the coarse time-step size $\Delta T$) is strictly larger for some $k\ge 1$ than for $k=0$.

To state the main results of this article (see Theorems~\ref{theo:error-implicit},~\ref{theo:error-exp1} and~\ref{theo:error-exp2} below for precise statements), let us assume that, for some $\alpha>0$, the covariance operator $Q$ satisfies a condition of the type
\[
\|(-A)^{\alpha-\frac12}Q^\frac12\|_{\mathcal{L}_2(H)}<\infty,
\]
where $\|\cdot\|_{\mathcal{L}_2(H)}$ denotes the Hilbert-Schmidt norm for operators from $H$ to $H$.

First, assume that the linear implicit Euler scheme is chosen as the coarse integrator, {\it i.e.} $e^{\Delta TA}$ is approximated by $(I-\Delta TA)^{-1}$. In addition, assume that $F=0$. Then the order of convergence of the parareal algorithm, with $k$ iterations, is essentially $\min(\alpha,k+1)$, and saturates at $\alpha$ when $k$ increases, see Theorem~\ref{theo:error-implicit} and numerical experiments in Section~\ref{sec:4-num}. In particular, if $\alpha\le 1$, which includes the space-time white noise case with essentially $\alpha=\frac14$, the application of the parareal algorithm is useless, since this order of convergence does not depend on $k$. The way the error behaves in terms of $k$ thus depends a lot on the regularity of the noise.

Second, assume that the exponential Euler scheme is chosen as the coarse integrator. Then it is proved that the order of convergence of the parareal algorithm, with $k$ iterations, is at least essentially of size $(k+1)\alpha$, if $\alpha\in(0,\frac12]$, see Theorem~\ref{theo:error-exp1}. Contrary to the first case, this order is linear in $k$, thus applying the parareal algorithm always reduces the computational cost, whatever the regularity of the noise. Numerical experiments, see Section~\ref{sec:5-num}, reveal that the order of convergence is sharp for $k=0$ (the parareal algorithm is not applied) and $k=1$ (one iteration is applied), but is larger when $k\ge 2$. In fact, the choice $k=1$ is optimal when considering the final computational cost. Theorem~\ref{theo:error-exp2} proves that for $\alpha\in(0,\frac14]$ and $k\ge 2$, the order of convergence is at least of size $2k\alpha$, which is indeed larger than $(k+1)\alpha$.

In conclusion, the parareal algorithm may offer an effective strategy to reduce computational cost for the simulation of trajectories of SPDEs. Several questions remain open: for instance, generalizations in higher dimension, algorithms for equations with multiplicative noise, or using other integrators, are left for future works.

This article is organized as follows. Precise assumptions on the operators $A$ and $F$, and on the covariance operator $Q$, are provided in Section~\ref{sec:2}. Section~\ref{sec:3} is devoted to introducing the parareal algorithm, to presenting the possible choices of coarse integrators, and to defining the error. The study of the behavior of the parareal algorithm when the linear implicit Euler scheme is chosen as the coarse integrator is provided in Section~\ref{sec:4}: more precisely, see Theorem~\ref{theo:error-implicit} for the theoretical error estimates and Section~\ref{sec:4-num} for the numerical experiments. The study of the behavior of the parareal algorithm when the exponential Euler scheme is chosen as the coarse integrator is provided in Section~\ref{sec:5}: more precisely, see Theorems~\ref{theo:error-exp1} and~\ref{theo:error-exp2} for the theoretical error estimates and Section~\ref{sec:5-num} for the numerical experiments.

\section{Setting}
\label{sec:2}

Let $H=L^2(0,1)$, with norm and inner product denoted by $|\cdot|$ and $\langle\cdot,\cdot\rangle$ respectively. In this work, stochastic evolution equations, with additive noise, of the following type are considered (see~\cite{DaPrato_Zabczyk:14},~\cite{LR}: given an initial condition $u(0)=u_0\in H$, 
\begin{equation}\label{eq:SPDE}
\begin{aligned}
du(t)=Au(t)dt+F(u(t))dt+dW^Q(t),
\end{aligned}
\end{equation}
where the solution $\bigl(u(t)\bigr)_{t\ge 0}$ is a stochastic process with values in $H$. The operator $A$ is a linear parabolic differential operator with homogeneous Dirichlet boundary conditions (Section~\ref{sec:ass_A}), the operator $F$ is a globally Lipschitz, non-linear operator (Section~\ref{sec:ass_F}), and $\bigl(W^Q(t)\bigr)_{t\ge 0}$ is a $Q$-Wiener process. Appropriate assumptions to justify the global well-posedness of~\eqref{eq:SPDE} are introduced below.

In the sequel, the initial condition $u_0$ is assumed to be deterministic, however the extension to a random initial condition $u_0$ (independent of the Wiener process, and satisfying appropriate moment conditions) is straightforward by a conditioning argument.

\subsection{Linear operator $A$}\label{sec:ass_A}

The linear operator $A$ is defined as the unbounded linear operator on $H=L^2(0,1)$, such that
\[
\begin{cases}
D(A)=\mathbb{H}^2(0,1)\cap \mathbb{H}_0^1(0,1),\\
Au=u'',~u\in D(A).
\end{cases}
\]
In other words, $A$ is the Laplace operator with homogeneous Dirichlet boundary conditions. Recall that $A$ is an unbounded, self-adjoint, linear operator, and that $Ae_p=-\lambda_p$ for all $p\in\N$, where the eigenvalues are given by $\lambda_p=(\pi p)^2$, and the eigenfunctions $e_p=\sqrt{2}\sin\bigl(p\pi\cdot\bigr)$ form a complete orthonormal system of $H$.

The linear operator $A$ generates an analytic and strongly continuous semigroup on $H$, denoted by $\bigl(e^{tA}\bigr)_{t\ge 0}$. Note that for all $u\in H$, one has
\[
e^{tA}u=\sum_{p\in\N}e^{-\lambda_p t}\langle u,e_p\rangle e_p.
\]

For any $\alpha\in[0,1]$ and $u\in H$, let
\[
|u|_\alpha^2=\sum_{p\in\N}\lambda_p^{2\alpha}\langle u,e_p\rangle^2\in[0,\infty]. 
\]
For $u\in D((-A)^\alpha)=\{u\in H:|u|_\alpha<\infty\}$, set $(-A)^\alpha u=\sum_{p\in\N}\lambda_p^\alpha \langle u,e_p\rangle e_p\in H$, and note that $|u|_\alpha=|(-A)^\alpha u|$. In addition, for any $\alpha\in[0,1]$ and $u\in H$, let $(-A)^{-\alpha}u=\sum_{p\in\N}\lambda_p^{-\alpha}\langle u,e_p\rangle e_p\in H$, and $|u|_{-\alpha}^2=\sum_{p\in\N}\lambda_p^{-2\alpha}\langle u,e_p\rangle^2$.

Regularization properties of the semigroup $\bigl(e^{tA}\bigr)_{t\ge 0}$ are stated in Proposition~\ref{propo:A1} below. The following notation is used. First, $\mathcal{L}(H)$ is the space of linear bounded operators from $H$ to $H$, with the operator norm denoted by $\|\cdot\|_{\mathcal{L}(H)}$. Second, $\mathcal{L}_2(H)$ is the space of Hilbert-Schmidt operators from $H$ to $H$, with the Hilbert-Schmidt norm denoted by $\|\cdot\|_{\mathcal{L}_2(H)}$.

\begin{propo}\label{propo:A1}
For all $t\ge 0$, $\|e^{tA}\|_{\mathcal{L}(H)}\le e^{-\lambda_1 t}$. Moreover, for all $\alpha\in[0,1]$, there exists $C_\alpha\in(0,\infty)$ such that for all $t\in(0,\infty)$,
\[
\|(-A)^{\alpha}e^{tA}\|_{\mathcal{L}(H)}\le C_\alpha \min(t,1)^{-\alpha},~\|(-A)^{-\alpha}(e^{tA}-I)\|_{\mathcal{L}(H)}\le C_\alpha \min(t,1)^\alpha.
\]
\end{propo}

\subsection{Nonlinear operator $F$}\label{sec:ass_F}

The analysis of the rate of convergence for parareal algorithm below proceeds in a simplified, abstract, framework, whereas numerical experiments are performed in the more natural framework of Nemytskii operators. The abstract framework does not encompass this case. Indeed, the treatment of Nemytskii would require the introduction of further concepts (such as $\gamma$-Radonifying operators, in order to work in Banach spaces $L^p(0,1)$). Instead of increasing the technical level of the presentation, the choice made for this article is to study the main features of the parareal algorithms applied to SPDEs in a more pedagogical way owing to the simplified framework.

Let us first state the assumptions on the nonlinear operator $F$ which are employed for the theoretical analysis.

\begin{as}\label{ass:F}
The nonlinear operator $F:H\to H$ is globally Lipschitz continuous, and is twice Fr\'echet differentiable, with bounded first and second order derivatives. Moreover, for any $\alpha\in(0,\frac12)$ and any arbitrarily small $\kappa\in(0,\frac12-\alpha)$, there exists $C_{F,\alpha,\kappa}\in(0,\infty)$ such that, for all $u\in D((-A)^{\alpha+\kappa})$ and $h\in H$,
\[
|DF(u).h|_{-\alpha}\le C_{F,\alpha,\kappa}\bigl(1+|u|_{\alpha+\kappa}\bigr)|h|_{-\alpha},
\]
and for all $u,h\in D((-A)^\alpha)$,
\[
|DF(u).h|_{\alpha-\kappa}\le C_{F,\alpha,\kappa}\bigl(1+|u|_\alpha\bigr)|h|_{\alpha}.
\]

Finally, for all $u_1,u_2,h\in D((-A)^{\alpha+\kappa})$,
\[
\big|\bigl(DF(u_2)-DF(u_1)\bigr).h\big|_{-\alpha}\le C_{F,\alpha,\kappa}\bigl(1+|u_1|_{\alpha+\kappa}+|u_2|_{\alpha+\kappa}\bigr)|u_2-u_1|_{-\alpha}|h|_{\alpha+\kappa}.
\]

\end{as}

Let us now recall that a Nemytskii operator $F:H\to H$ is defined such that $F(u)=f\circ u$ for all $u\in H$, where $f:\R\to\R$ is a real-valued mapping, assumed to be at least globally Lipschitz continuous. As explained above, even if $f$ is assumed of class $\mathcal{C}^2$ with bounded first and second order derivatives, the associated nonlinear Nemytskii operator $F$ does not satisfy the conditions of Assumption~\ref{ass:F}. Indeed, the appropriate generalization requires estimates in $L^p(0,1)$ spaces, for $p\in(2,\infty)$ (using H\"older inequality).

\subsection{Wiener process}\label{sec:ass_Q}

Let $\bigl(\Omega,\mathcal{F},\mathbb{P}\bigr)$ be a probability space, equipped with a filtration $\bigl(\mathcal{F}_t\bigr)_{t\ge 0}$ satisfying the usual conditions. The expectation operator is denoted by $\E[\cdot]$.

Let $\left(\bigl(\beta_p(t)\bigr)_{t\ge 0}\right)_{p\in\N}$ denote a sequence of independent standard real-valued Wiener processes, and let $\bigl(\epsilon_p\bigr)_{p\in\N}$ be a complete orthonormal system of $H$, and $\bigl(\gamma_p\bigr)_{p\in\N}$ be a sequence of nonnegative real numbers. The cylindrical Wiener process is defined as
\[
W(t)=\sum_{p\in\N}\beta_p(t)\epsilon_p.
\]
The $Q$-Wiener process is then defined as
\[
W^Q(t)=\sum_{p\in\N}\sqrt{\gamma_p}\beta_p(t)\epsilon_p,
\]
and can be written as $W^Q(t)=Q^{\frac12}W(t)$, where the linear self-adjoint operators $Q^{\frac12}$ and $Q$ satisfy
\[
Q^{\frac12}u=\sum_{p\in\N}\sqrt{\gamma_p}\langle x,\epsilon_p\rangle \epsilon_p,~Qu=\sum_{p\in\N}\gamma_p\langle x,\epsilon_p\rangle \epsilon_p,\quad\forall~ u\in H.
\]
Note that the $Q$-Wiener process $W^Q(t)$ takes values in $H$ if and only if $Q$ is a trace-class linear operator, {\it i.e.}, ${\rm Tr}(Q)=\|Q^{\frac12}\|_{\mathcal{L}_2(H)}^2=\sum_{p\in\N}\gamma_p<\infty$.

Assumption~\ref{ass:Q} states the conditions on $Q$ required to ensure the well-posedness of~\eqref{eq:SPDE}.
\begin{as}\label{ass:Q}
Assume that there exists $\alpha>0$ such that $\|(-A)^{\alpha-\frac{1}{2}}Q^{\frac12}\|_{\mathcal{L}_2(H)}<\infty$.
\end{as}
Define the parameter $\al$ as follows:
\[
\al=\sup\left\{\alpha\in(0,\infty),~\|(-A)^{\alpha-\frac{1}{2}}Q^{\frac12}\|_{\mathcal{L}_2(H)}<\infty\right\},
\]
then $\al>0$ if and only if Assumption~\ref{ass:Q} is satisfied. For instance, if $Q=I$ (space-time white noise), then $\al=\frac14$. If $Q$ is a trace-class operator, then $\al=\frac12$. In the article, we are mostly interested in the regime $\al\in(0,\frac12]$.

The numerical experiments are performed with the following example: for all $p\in\N$, $\epsilon_p=e_p$ (thus the operators $A$ and $Q$ commute), and $\gamma_p=\lambda_p^{\frac12-2\al}$ (observe that in that case the notation is consistent with the definition of $\al$ in the general case).

To conclude this section, let us introduce the following notation: if $U$ is a $H$-valued random variable, for all $\alpha\in[0,1]$, and $q\in\N$,
\[
\vvvert U\vvvert_{q}=\bigl(\E[|U|^q]\bigr)^{\frac1q},~\vvvert U\vvvert_{q,\alpha}=\bigl(\E[|U|_\alpha^q\bigr)^{\frac1q}.
\]

\subsection{Well-posedness and regularity properties}

Solutions of~\eqref{eq:SPDE} are understood in the mild sense: for all $t\ge 0$,
\begin{equation}\label{eq:SPDE_mild}
u(t)=e^{tA}u_0+\int_{0}^{t}e^{(t-s)A}F(u(s))ds+\int_0^t e^{(t-s)A}dW^Q(s).
\end{equation}

Under the assumptions stated above, this problem is globally well-posed. We quote without proof the following standard result.
\begin{propo}\label{propo:bound_u}
Let Assumption~\ref{ass:Q} be satisfied. For any initial condition $u_0\in H$, there exists a unique mild solution~\eqref{eq:SPDE_mild} of the SPDE~\eqref{eq:SPDE}. Moreover, for any $T>0$ and $q\in\N$, there exists $C_{T,q}\in(0,\infty)$ such that
\[
\underset{0\le t\le T}\sup~\vvvert u(t)\vvvert_q\le C_{T,q}\bigl(1+|u_0|\bigr).
\]

Moreover, for any $\alpha\in\bigl(0,\min(\al,\frac12)\bigr)$, there exists $C_{T,q,\alpha}\in(0,\infty)$ such that
\begin{align*}
\vvvert u(t)\vvvert_{q,\alpha} &\le C_{T,q,\alpha}\left(1+\min\bigl(|u_0|_{\alpha},t^{-\alpha}|u_0|\bigr)\right),~\forall~ t\in(0,T],\\
\vvvert u(t)-u(s)\vvvert_q &\le C_{T,q,\alpha}|t-s|^{\alpha}\left(1+\min\bigl(|u_0|_{\alpha},\min(t,s)^{-\alpha}|u_0|\bigr)\right),~\forall~ t,s\in(0,T].
\end{align*}
\end{propo}

\section{Parareal algorithms}
\label{sec:3}

\subsection{Fine and coarse integrators}

Let $T\in(0,\infty)$ be given. Introduce the so-called coarse and fine time-step sizes $\Delta T$ and $\delta t$. It is assumed that $T=N\Delta T$ and $\Delta T=J\delta t$, where $N$ and $J$ are integers. For all $n\in\{0,\ldots,N\}$ and $j\in\{0,\ldots,J\}$, let
\[
t_n=n\Delta T,\quad t_{n,j}=t_n+j\delta t=(nJ+j)\delta t.
\]

Note that the coarse and the fine integrators introduced below are random mappings. Precisely, for all $n\in\N$, the mappings $\G_n=\G(\cdot,t_n,t_{n+1})$ and $\F_n=\F(\cdot,t_n,t_{n+1})$ depend on the Wiener increments $\bigl(W^Q(t)-W^Q(t_n)\bigr)_{t_n\le t\le t_{n+1}}$.

\subsubsection{Coarse integrator}

The coarse integrator is a numerical scheme with time-step size $\Delta T$. In this work, it is defined as follows: for all $n\in\{0,\ldots,N-1\}$ and all $u\in H$, let
\begin{equation}\label{eq:coarse}
\G(u,t_n,t_{n+1})=\hat S_{\Delta T}u+\Delta T\hat S_{\Delta T}F(u)+\hat S_{\Delta T}\bigl(W^Q(t_{n+1})-W^Q(t_n)\bigr),
\end{equation}
with
\begin{itemize}
\item either $\hat{S}_{\Delta T}=e^{\Delta TA}$ (exponential Euler scheme),
\item or $\hat{S}_{\Delta T}=S_{\Delta T}=(I-\Delta TA)^{-1}$ (linear implicit Euler scheme).
\end{itemize}

The notation $\G_{\rm expo}$, resp. $\G_{\rm imp}$, is often used below, to refer to the coarse integrator with the exponential Euler scheme, resp. with the linear implicit Euler scheme. As will be seen below, these two coarse integrators have very different behaviors when applied to the SPDEs considered in this article.

\subsubsection{Fine integrator}

The fine integrator consists of $J$ steps of a numerical scheme with time-step size $\delta t$. In this work, this numerical scheme is obtained by the exponential Euler scheme. More precisely, introduce the auxiliary integrator $\F_{\rm aux}$: for all $n\in\{0,\ldots,N-1\}$, all $j\in\{0,\ldots,J-1\}$ and all $u\in H$,
\begin{equation*}\label{eq:coarse-fine}
\F_{\rm aux}(u,t_{n,j},t_{n,j+1})=e^{\delta tA}u+\delta te^{\delta t A}F(u)+e^{\delta t A}\bigl(W^Q(t_{n,j+1})-W^Q(t_{n,j})\bigr).
\end{equation*}
The fine integrator $\F$, at the coarse time scale, is then defined as follows:
\begin{equation}\label{eq:fine}
\F(\cdot,t_n,t_{n+1})=\F_{\rm aux}(\cdot,t_{n,J-1},t_{n,J})\circ\cdots\circ\F_{\rm aux}(\cdot,t_{n,0},t_{n,1}).
\end{equation}
In other words, the solution $v_{n,J}=\F(u,t_n,t_{n+1})$. is computed using the following recursion formula
\begin{align*}
v_{n,j+1}&=e^{\delta tA}v_{n,j}+\delta te^{\delta tA}F(v_{n,j})+e^{\delta tA}\bigl(W^Q(t_{n,j+1})-W^Q(t_{n,j})\bigr)\\
v_{n,0}&=u
\end{align*}
for $j\in\{0,\cdots,J-1\}$ with $J\delta t=\Delta T$.

\subsection{The parareal algorithm}

The initialization step of the parareal algorithm consists in applying the coarse integrator: for $n\in\{0,\ldots,N-1\}$
\begin{align*}
u_{n+1}^{(0)}&=\G(u_n^{(0)},t_n,t_{n+1}),\\
u_0^{(0)}&=u_0.
\end{align*}

Let $K\in\N$ denote the number of parareal iterations. Iterations for $k\in\{0,\ldots,K-1\}$ are defined as follows: given the values $\bigl(u_m^{(k)}\bigr)_{0\le m\le N}$ at iteration $k$, then compute, for all $n\in\{0,\ldots,N-1\}$,
\begin{equation}\label{eq:parareal}
\begin{aligned}
u_{n+1}^{(k+1)}&=\G(u_{n}^{(k+1)},t_n,t_{n+1})+\F(u_{n}^{(k)},t_n,t_{n+1})-\G(u_{n}^{(k)},t_n,t_{n+1}),\\
u_{0}^{(k)}&=u_0.
\end{aligned}
\end{equation}

The core of the approach lies in the ability, at each iteration in $k$, to perform in parallel the computations in~\eqref{eq:parareal} for different values of $n$, hence the terminology of ``parareal algorithms'' for ``parallelization in real time''.

\subsection{The reference solution}

The reference solution is defined using the fine integrator:
\begin{equation}\label{eq:ref}
\begin{aligned}
u_{n+1}^{\rm ref}&=\F(u_n^{\rm ref},t_n,t_{n+1}),\\
u_0^{\rm ref}&=u_0.
\end{aligned}
\end{equation}
Observe that following~\eqref{eq:fine}, this reference solution is in fact defined in terms of the integrator $\F_{\rm aux}$, applied with the fine time-step size $\delta t$. Precisely, for all $n\in\{0,\ldots,N-1\}$, one has $u_{n+1}^{\rm ref}=v_{n,J}^{\rm ref}$ defined by
\begin{align*}
v_{n,j+1}^{\rm ref}&=\F_{\rm aux}\bigl(v_{n,j}^{\rm ref},t_{n,j},t_{n,j+1}\bigr),~j\in\{0,\ldots,J-1\}\\
v_{n,0}^{\rm ref}&=u_n^{\rm ref}.
\end{align*}
This may be rewritten as follows: $u_{n}^{\rm ref}=v_{nJ}^{\rm ref}$, where for all $\ell\in\{0,\ldots,NJ-1\}$,
\[
v_{\ell+1}^{\rm ref}=\F_{\rm aux}\bigl(v_{\ell}^{\rm ref},\ell\delta t,(\ell+1)\delta t\bigr).
\]
In addition, note that
\[
v_{\ell}^{\rm ref}=e^{\ell\delta tA}u_0+\delta t\sum_{l=0}^{\ell-1}e^{(l-\ell)\delta tA}F(v_{l}^{\rm ref})+\sum_{l=0}^{\ell-1}e^{(l-\ell)\delta tA}\bigl(W^Q((l+1)\delta t)-W^Q(l\delta t)\bigr).
\]

To conclude this section, we state without proof two standard results concerning, first, the qualitative properties of the reference solution (moment estimates), second, the rate of convergence of the error $u_n^{\rm ref}-u(t_n)$.

\begin{propo}\label{propo:bound_ref}
Let $T>0$, $\alpha\in\bigl(0,\min(\al,\frac12)\bigr)$ and $q\in\N$. There exists $C_{T,q,\alpha}\in(0,\infty)$ such that for all $u_0\in D((-A)^\alpha)$,
\[
\underset{0\le n\le N-1}{\sup}\vvvert u_n^{\rm ref}\vvvert_{q,\alpha}\le C_{T,q,\alpha}\bigl(1+|u_0|_{\alpha}\bigr),
\]
and such that the following error estimate holds true:
\[
\underset{0\le n\le N-1}{\sup}\vvvert u_n^{\rm ref}-u(t_n)\vvvert_{q}\le C_{T,q,\alpha}\delta T^{\alpha}\bigl(1+|u_0|_{\alpha}\bigr).
\]
\end{propo}

\begin{rem}
The assumption that $u_0\in D((-A)^\alpha)$ may be weakened using Proposition~\ref{propo:A1}.
\end{rem}

\begin{rem}
Since noise is additive in the SPDE~\eqref{eq:SPDE}, the order of convergence in Proposition~\ref{propo:bound_ref} may be larger than $\frac12$ when $\al$ is sufficiently large. This type of estimate is not considered in this article.
\end{rem}

\subsection{Error and residual operators}

In the implementation of parareal algorithms, the reference solution $u_n^{\rm ref}$ defined above is not computed in practice. Instead, the quantity $u_n^{(k)}$ defined in~\eqref{eq:parareal} is computed. To estimate the error between $u_n^{(k)}$ and $u(t_n)$, due to Proposition~\ref{propo:bound_ref}, it is sufficient to study the error between $u_n^{(k)}$ and $u_n^{\rm ref}$. Observe that the parareal solution $u_n^{(k)}$ may be computed with a lower computational cost than the reference solution $u_n^{\rm ref}$, using parallel computations in~\eqref{eq:parareal} (expect for the initialization). Error estimates are required to determine the choice of time-step sizes $\Delta T$ and $\delta t$, and of the number of parareal iterations $K$, to achieve a given error criterion, with minimal computational cost.

For all $n\in\{0,\ldots,N\}$ and $k\in\{0,\ldots,K\}$ (where $K$ is the number of parareal iterations), let the error be defined by
\begin{equation}\label{eq:error_def}
\epsilon_n^{(k)}:=u_n^{(k)}-u_n^{\rm ref}.
\end{equation}
Note that $\epsilon_0^{(k)}=0$ for all $k\in\N_0$.

Moreover, by construction, one get $\epsilon_n^{(k)}=0$ for all $n\le k$, which indicates that the numerical solution $\{u_n^{(k)}\}_{n=0,\cdots,N}$ will definitely converge to the reference solution $\{u_n^{\rm ref}\}_{n=0,\cdots,N}$, if the iterated number $K$ is sufficiently large, {\it i.e.} $K\ge N$.
However, to get a speedup, in practice, the iterated number $K$ will be chosen significantly smaller than $N$, which is further discussed in Section \ref{sec:cost}.

It is convenient to introduce the residual operators defined by
\begin{equation}\label{eq:residual_def}
\RR_n(u):=\F_n(u)-\G_n(u),
\end{equation}
for all $n\in\{0,\ldots,N-1\}$, where the notation $\F_n(u):=\F(u,t_n,t_{n+1})$ and $\G_n(u):=\G(u,t_n,t_{n+1})$ is used.

Then the error defined by~\eqref{eq:error_def} satisfies the recursion formula, where the residual operators $\RR_n$ defined by~\eqref{eq:residual_def} appear:
\begin{equation}\label{eq:error}
\begin{aligned}
\epsilon_{n+1}^{(k+1)}&=\G_n(u_{n}^{(k+1)})+\F_n(u_{n}^{(k)})-\G_n(u_{n}^{(k)})-\F_n(u^{\rm ref}_{n})\\
&=\G_n(u_{n}^{(k+1)})-\G_n(u_{n}^{\rm ref})+\RR_n(u_{n}^{(k)})-\RR_n(u_{n}^{\rm ref})\\
&=\hat{S}_{\Delta T}\epsilon_{n}^{(k+1)}+\Delta T\hat{S}_{\Delta T}\left[F(u_{n}^{(k+1)})-F(u^{\rm ref}_{n})\right]+\RR_{n}(u_{n}^{(k)})-\RR_{n}(u_{n}^{\rm ref}),
\end{aligned}
\end{equation}
where the linear operator $\hat S_{\Delta T}$ depends on the choice of the coarse integrator, see~\eqref{eq:coarse}.

Up to this point, the choice of the coarse integrator plays no role in the presentation. The major finding of this article is that the behavior of the parareal algorithm applied for SPDEs~\eqref{eq:SPDE} differs when choosing the exponential Euler scheme or the linear implicit Euler scheme as the coarse integrator. Indeed, the theoretical results and the numerical experiments reveal that, as the number of parareal iterations $k$ increases, the evolution of the order of convergence of the error $\epsilon_n^{(k)}$ has a different behavior depending on the choice of coarse integrator.

\subsection{Analysis of the computational cost}\label{sec:cost}

The objective of this section is to compare the costs for computing $u_n^{(k)}$ using the parareal algorithm~\ref{eq:parareal}, and for computing the reference solution $u_n^{\rm ref}$.

The computational advantage of using the parareal algorithm is due to the possibility to compute the quantities $u_{n}^{(k+1)}$ in parallel, for fixed $k\ge 0$, see~\eqref{eq:parareal}. Let $N_{\rm proc}$ denote the number of available processors.

Let $T\in(0,\infty)$ denote the final time, and consider $n=N$ such that $N\Delta T=T$. Denote by $\tau_{\G}$ the computational time for one evaluation of $\G(u,t_n,t_{n+1})$. It is assumed that $\tau_{\G}$ does not depend on $\Delta T$, $n\in\N_0$ and on $u\in H$. Denote also by $\tau_{\F,{\rm aux}}$ denote the computational time for one realization of $\F_{\rm aux}(u,t_{n,j},t_{n,j+1})$. Then the computational time for one realization of $\F(\cdot,t_n,t_{n+1})$, denoted by $\tau_{\F}$, is 
\[
\tau_{\F}=J\tau_{\F,{\rm aux}}=\frac{\Delta T \tau_{\F,{\rm aux}}}{\delta t}.
\]
It is also assumed that $\tau_{\F,{\rm aux}}(u,t_{n,j},t_{n,j+1})$ does not depend on $\delta t$, $n,j$ and on $u$.

\subsubsection{Parareal algorithm}

For the initialization step, the computational cost is equal to $N\tau_{\G}$, since at this stage no parallelization procedure is applied.

For each iteration of the algorithm, observe that in~\eqref{eq:parareal}, the third term $\G(u_n^{(k)},t_n,t_{n+1})$ has already been computed at the previous iteration, and that the values of the second term $\F(u_n^{(k)},t_n,t_{n+1})$ may be computed in parallel. A sequential computation remains to be done, thus the computational cost of one iteration of the parareal algorithm is
\[
N\bigl(\tau_{\G}+\frac{\tau_{\F}}{N_{\rm proc}}\bigr).
\]
If $K$ iterations of the parareal algorithm are performed, the associated computational cost is thus equal to 
\[
{\rm Cost}^{\rm parareal}=(K+1)\frac{T}{\Delta T}\tau_{\G}+K\frac{T}{\delta t}\frac{\tau_{\F,{\rm aux}}}{N_{\rm proc}}.
\]

\subsubsection{Reference solution}

The reference solution $u_n^{\rm ref}$ is computed using the fine integrator $\F_{\rm aux}$ with time-step size $\delta t$, see~\eqref{eq:ref}. The associated computational cost is equal to
\[
{\rm Cost}^{\rm ref}=\frac{T}{\Delta T}\tau_{\F}=\frac{T}{\delta t}\tau_{\F,{\rm aux}}.
\]

\subsubsection{Efficiency}

The efficiency of the parareal algorithm, compared with a direct simulation using the fine integrator only, is thus studied in terms of the ratio
\[
\mathcal{E}=\frac{{\rm Cost}^{\rm ref}}{{\rm Cost}^{\rm parareal}}=\frac1{\frac{K}{N_{\rm proc}}+(K+1)\frac{\delta t}{\Delta T}\frac{\tau_{\G}}{\tau_{\F,\rm aux}}}.
\]

Note that since the efficiency never goes to infinity as $\Delta T$ goes to $0$, the parareal algorithm does not improve the rate of convergence with respect to the time-step size. Instead, the parareal algorithm may improve the computational efficiency. To go further in the analysis of the efficiency of the parareal algorithm, it is essential to study the rate of convergence of the error $\epsilon_n^{(k)}$ in terms of $\Delta T$.

On the one hand, assume that the rate of convergence does not depend on $k$. Then, to balance the errors $u_n^{(k)}-u_n^{\rm ref}$ and $u_n^{\rm ref}-u(n\Delta T)$, it is necessary to choose $\Delta T$ and $\delta t$ of the same size, hence
\[
\mathcal{E}=\frac{{\rm Cost}^{\rm ref}}{{\rm Cost}^{\rm parareal}}=\frac1{\frac{K}{N_{\rm proc}}+C(K+1)\frac{\tau_{\G}}{\tau_{\F,\rm aux}}}.
\]
To maximize the efficiency $\mathcal{E}$ above, the optimal choice is apparently $K=0$: parareal iterations increase the computational cost (linearly in $K$), in spite of the use of parallelization.

On the other hand, assume that the rate of convergence of $\epsilon_n^{(k)}$ is strictly larger than the rate for $\epsilon_{n}^{(0)}$. Then the time-step sizes are chosen such that $\delta t={\rm o}(\Delta T)$ as $\Delta T\to 0$, and the efficiency is then of size $\frac{N_{\rm proc}}{K}$. This means that parallelization ($N_{\rm proc}\ge 2$) reduces the cost, and that the optimal choice is $K=1$.

In Sections~\ref{sec:4} and~\ref{sec:5}, the rates of convergence of the error $\epsilon_n^{(k)}$ with respect to $\Delta T$ are studied, depending on the choice of the coarse integrator, and on the regularity of the noise.

\section{Linear-implicit Euler scheme as the coarse integrator}
\label{sec:4}

The objective of this section is to prove that, when the coarse integrator is chosen as the linear implicit Euler scheme, {\it i.e.} $\G=\G_{\rm imp}$ with $\hat S_{\Delta T}=S_{\Delta T}$, then the behavior of the parareal algorithm depends a lot on the regularity of the noise. More precisely, it is proved that the order of convergence of the error $\epsilon_n^{(k)}$ to $0$ cannot exceed $\al$, and in particular saturates when $k$ increases. Essentially, the order of convergence (in the framework studied below) is equal to $\min(\al,k+1)$. The theoretical and numerical results are consistent, and show that the rates obtained by the theoretical analysis are sharp.

For the theoretical analysis developped in this section, the framework is as follows. First, the initial condition $u_0=0$ and the nonlinear operator $F=0$ are set equal to $0$. Second, the covariance $Q$ of the noise commutes with $A$, {\it i.e.} $Qe_p=\gamma_pe_p$, for all $p\in\N$, and eigenvalues satisfy $\gamma_p=\lambda_p^{\frac12-2\al}$.

In this case, the solution is a Gaussian process, and $u(t)=\int_{0}^{t}e^{(t-s)A}dW^Q(s)$. Moreover, the recursion formula~\eqref{eq:error} for the error yields the equalities
\[
\epsilon_{n}^{(k+1)}
=S_{\Delta T}\epsilon_{n-1}^{(k+1)}+\bigl(e^{\Delta TA}-S_{\Delta T}\bigr)\epsilon_{n-1}^{(k)}
=\sum_{m=0}^{n-1}S_{\Delta T}^{n-1-m}\bigl(e^{\Delta TA}-S_{\Delta T}\bigr)\epsilon_{m}^{(k)}.
\]

\subsection{Theoretical error estimates}
\label{sec:4-error}

The main theoretical result of this section is Theorem~\ref{theo:error-implicit}.
\begin{theo}\label{theo:error-implicit}
Assume that $\gamma_p=\lambda_p^{\frac12-2\al}$, with $\al>0$. Let $T\in(0,\infty)$ and $k\in\N_0$.
\begin{itemize}
\item If $k+1<\al$, then there exists $C_{T,k,\al}\in(0,\infty)$ such that
\[
\sup_{n\Delta T\le T}\vvvert \epsilon_n^{(k)}\vvvert_2\le C_{T,k,\al}\Delta T^{k+1}.
\]
\item If $k+1\ge \al$, then for all $\alpha\in(0,\al)$, there exists $C_{T,k,\alpha}\in(0,\infty)$ such that
\[
\sup_{n\Delta T\le T}\vvvert \epsilon_n^{(k)}\vvvert_2\le C_{T,k,\alpha}\Delta T^{\alpha}.
\]
\end{itemize}
\end{theo}

In particular, for $\al=\frac14$ ($Q=I$, space-time white noise), or $\al=\frac12$ ($Q$ is trace-class), the rate of convergence does not depend on $k$, and performing parareal iterations does not increase the order of convergence, see Section~\ref{sec:cost}.

\begin{proof}

For all $p\in\N$, let $\epsilon_{n}^{(k)}(p)=\langle \epsilon_n^{(k)},e_p\rangle$ denote the $p$-th component of the error $\epsilon_n^{(k)}$. Then the expression above is rewritten as
\[
\epsilon_n^{(k+1)}(p)=\sum_{m=0}^{n-1}V(-\lambda_p\Delta T)^{n-1-m}R(-\lambda_p\Delta T)\epsilon_{m}^{(k)}(p),
\]
where, for $z\in(-\infty,0]$, one has $V(z):=\frac{1}{1-z}$ and $R(z):=e^z-\frac{1}{1-z}$. The inequality $|R(z)|\le 1\wedge |z|^2$ yields
\[
\bigl(\E|\epsilon_{n}^{(k+1)}(p)|^2\bigr)^{\frac12} \le C\bigl(1\wedge (\lambda_p\Delta T)^2\bigr)\left(\sum_{m=0}^{n-1}\frac{1}{(1+\lambda_p\Delta T)^{n-1-m}}\right)\underset{0\le m\le n}{\sup}\bigl(\E|\epsilon_{n}^{(k)}(p)|^2\bigr)^{\frac12}.
\]
Observe that
\[
\sum_{m=0}^{n-1}\frac{1}{(1+\lambda_p\Delta T)^{n-1-m}}\le \sum_{m=0}^{\infty}\frac{1}{(1+\lambda_p\Delta T)^{m}}=\frac{1+\lambda_p\Delta T}{\lambda_p\Delta T}.
\]
Thus for all $n$ and all $k$, one obtains (using a recursion argument)
\begin{align*}
\bigl(\E|\epsilon_{n}^{(k)}(p)|^2\bigr)^{\frac12}
\le& C\bigl(1\wedge (\lambda_p\Delta T)\bigr)\underset{0\le m\le n}{\sup}\bigl(\E|\epsilon_{m}^{(k-1)}(p)|^2\bigr)^{\frac12}\\
\le& C^k\bigl(1\wedge (\lambda_p\Delta T)\bigr)^k\underset{0\le m\le n}{\sup}\bigl(\E|\epsilon_{m}^{(0)}(p)|^2\bigr)^{\frac12}.
\end{align*}
It remains to study the error $\epsilon_n^{(0)}$ at the initialization step. One has the identity
\[
\epsilon_n^{(0)}=\sum_{m=0}^{n-1}\left(S_{\Delta t}^{n-m}-S((n-m)\Delta T)\right)\bigl(W^Q((m+1)\Delta T)-W^Q(m\Delta T)\bigr).
\]
Let us prove the following claim: there exists $C\in(0,\infty)$ such that
\begin{equation}\label{eq:claim}
\underset{n\in\N}{\sup}\E|\epsilon_{n}^{(0)}(p)|^2\le \frac{C\gamma_p}{\lambda_p}\bigl(1\wedge (\lambda_p\Delta T)\bigr)^2.
\end{equation}
Using the It\^o isometry formula, and the fact $b^m-a^m\le mb^{m-1}(b-a)$ for all $0\le a\le b$ and $m\in\N$, (with $b=V(-\lambda_p(\Delta T)$ and $a=e^{-\lambda_p\Delta T}$), one obtains
\begin{align*}
\E|\epsilon_n^{(0)}(p)|^2&=\gamma_p\Delta T\sum_{m=0}^{n-1}\left(\frac{1}{(1+\lambda_p\Delta T)^{(n-m)}}-e^{-\lambda_p(n-m)\Delta T}\right)^2\\
&=\gamma_p\Delta T\sum_{m=1}^{n}\left(\frac{1}{(1+\lambda_p\Delta T)^{m}}-e^{-\lambda_pm\Delta T}\right)^2\\
&\le \gamma_p\Delta T\sum_{m=1}^{n}\left(\frac{1}{(1+\lambda_p\Delta T)^{m}}-e^{-\lambda_pm\Delta T}\right)\frac{1}{(1+\lambda_p\Delta T)^{m}}\\
&\le C\gamma_p\Delta T\sum_{m=1}^{n}\frac{m}{(1+\lambda_p\Delta T)^{(2m-1)}}\bigl(1\wedge (\lambda_p\Delta T)^2\bigr)\\
&\le C\gamma_p\sum_{m=1}^{\infty}\frac{m\lambda_p\Delta T}{(1+\lambda_p\Delta T)^{(m-1)}} \frac{\bigl(1\wedge (\lambda_p\Delta T)^2\bigr)}{\lambda_p(1+\lambda_p\Delta T)}\\
&\le C\gamma_p\frac{\bigl(1\wedge (\lambda_p\Delta T)^2\bigr)}{\lambda_p}.
\end{align*}
This concludes the proof of the claim~\eqref{eq:claim}. Then, using the expression $\gamma_p=\lambda_p^{\frac12-2\al}$ of the eigenvalues of the covariance operator $Q$, one obtains
\begin{align*}
\vvvert\epsilon_n^{(k)}\vvvert_2&=\left(\sum_{p=1}^{\infty}\E|\epsilon_n^{(k)}(p)|^2\right)^{\frac12}\\
&\le C_k\left(\sum_{p=1}^{\infty}\frac{1}{\lambda_p^{\frac12+2\alpha}}\bigl(1\wedge (\lambda_p\Delta T)\bigr)^{2(k+1)}\right)^{\frac12}.
\end{align*}
It remains to identify the orders of convergence. On the one hand, assume that $k+1<\al$. Using $\bigl(1\wedge (\lambda_p\Delta T)\bigr)^{2(k+1)}\le (\lambda_p\Delta T)^{2(k+1)}$ yields
\[
\vvvert\epsilon_n^{(k)}\vvvert_2
\le C\Delta T^{k+1}\left(\sum_{p=1}^\infty\lambda_p^{-\frac12-2(\al-k-1)}\right)^\frac12\le C_{k,\al}\Delta T^{k+1},
\]
since $\al-k-1>0$.

On the other hand, assume that $k+1\ge \al$, and let $\alpha\in(0,\al)$. Using the inequality $\bigl(1\wedge (\lambda_p\Delta T)\bigr)^{2(k+1)}\le (\lambda_p\Delta T)^{2\alpha}$ yields
\[
\vvvert\epsilon_n^{(k)}\vvvert_2
\le C\Delta T^{\alpha}\left(\sum_{p=1}^\infty\lambda_p^{-\frac12-2\al+2\alpha}\right)^\frac12\le C_{k,\alpha}\Delta T^{\alpha}.
\]
This concludes the proof of Theorem~\ref{theo:error-implicit}.
\end{proof}

\subsection{Numerical experiments}\label{sec:4-num}

The objective of this section is to demonstrate that the orders of convergence obtained in Theorem~\ref{theo:error-implicit} are sharp. In addition, experiments in the semilinear case ($F\neq 0$) are also provided.

First, the SPDE $du(t)=Au(t)+dW^Q(t)$, with $u(0)=0$, is considered, where the covariance $Q$ is given as above ($Qe_p=\gamma_pe_p$, with $\gamma_p=\lambda_p^{\frac12-2\al}$). Spatial discretization is performed using finite differences, with mesh size $h=0.01$. In addition, the noise is truncated, {\it i.e.} the $Q$-Wiener process $W^Q(t)$ is replaced by $\sum_{p=1}^{P}\gamma_p^{\frac12}\beta_p(t)e_p$, with $P=100$. Numerical parameters are chosen as follows: the final time is $T=1$, the fine time-step size is $\delta t=2^{-13}$, and the coarse time-step size is $\Delta T=J\delta t$ with $J=2^{j}$, $j=4,\cdots,9$. An average over $M=100$ independent Monte-Carlo samples is used to approximate the expectations.

First, Figure~\ref{fig:imp-1} reports numerical simulations for $\al=0.25$ (space-time white noise) and $\al=0.5$ (trace-class noise). For all values of $k$, the order of convergence is equal to $\al$, as predicted by Theorem~\ref{theo:error-implicit}.

\begin{figure}[h]
\centering
  \includegraphics[width=7.6cm]{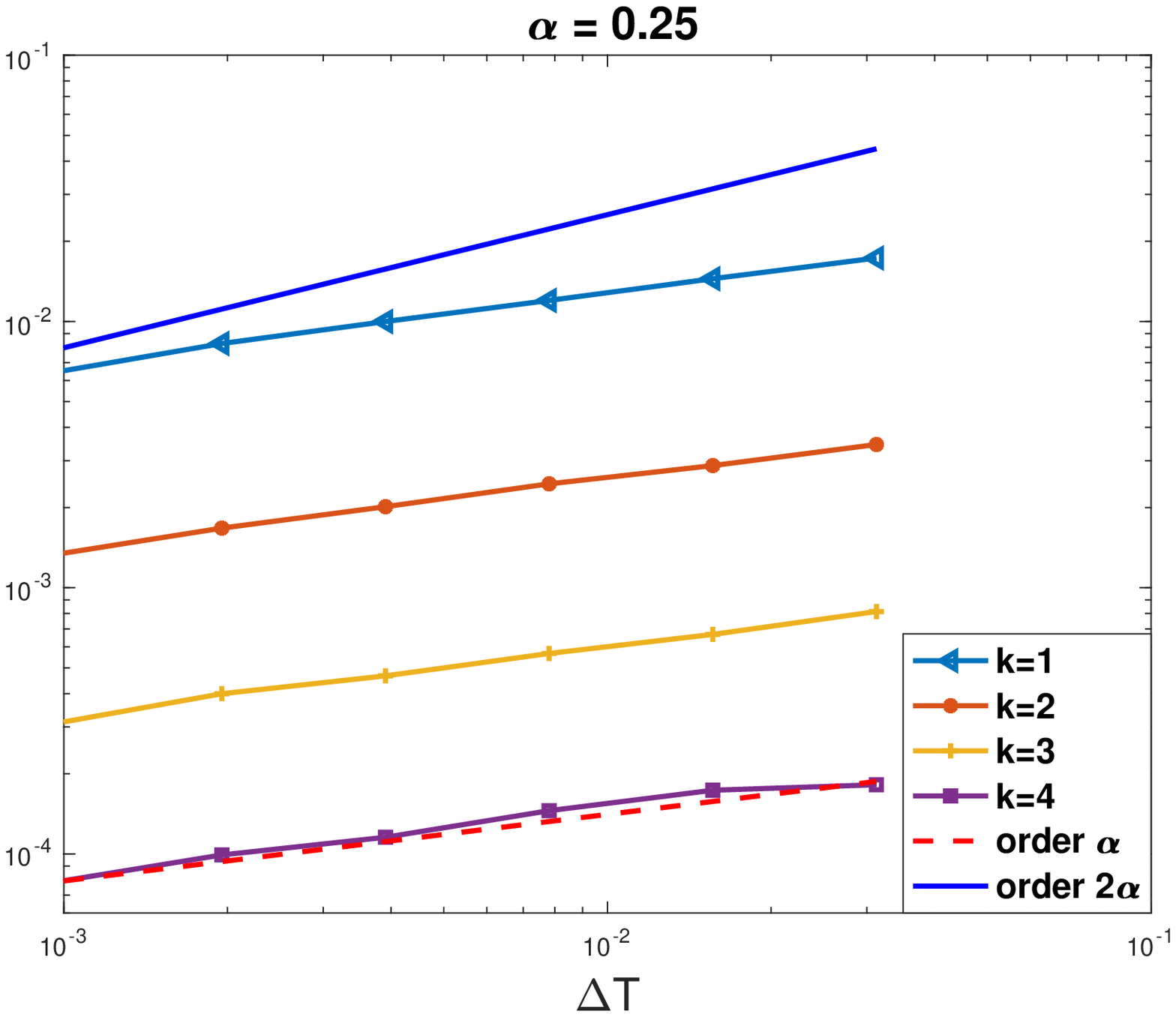}
  \includegraphics[width=7.6cm]{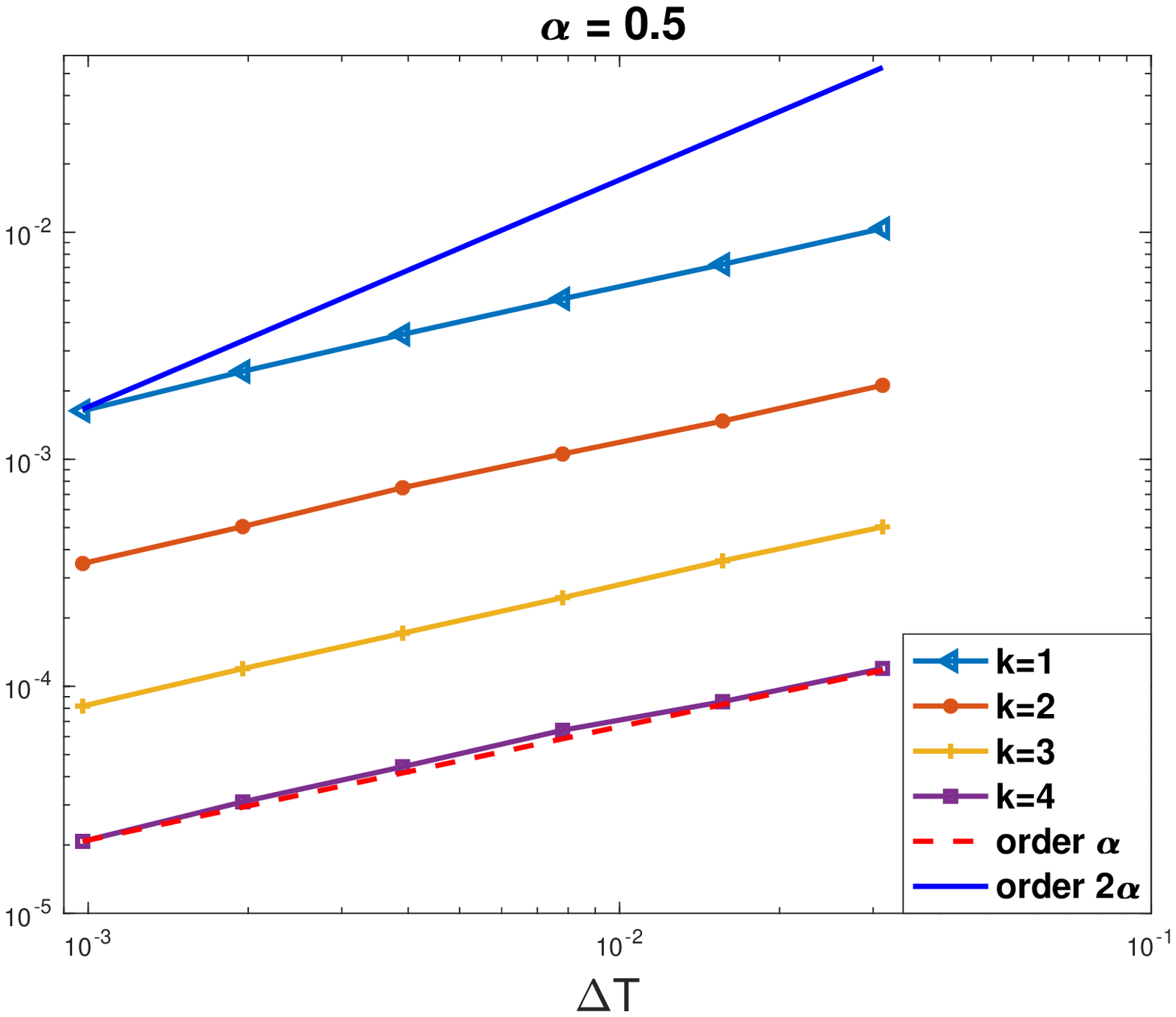}
  \caption{
Orders of convergence of the error with respect to $\Delta T$, for $\al=0.25$ (left) and $\al=0.5$ (right), for different values of $k\in\{1,2,3,4\}$, in the linear implicit Euler scheme case.
}
  \label{fig:imp-1}
\end{figure}

Second, Figures~\ref{fig:imp-2} (two fixed values of $k$ and $\al$ varies) and~\ref{fig:imp-3} (two fixed values of $\al$ and $k$ varies) allow us to check that the orders of convergence in Theorem~\ref{theo:error-implicit} are sharp: the order is indeed equal to $\min(\al,k+1)$.

\begin{figure}[h]
\centering
  \includegraphics[width=7.6cm]{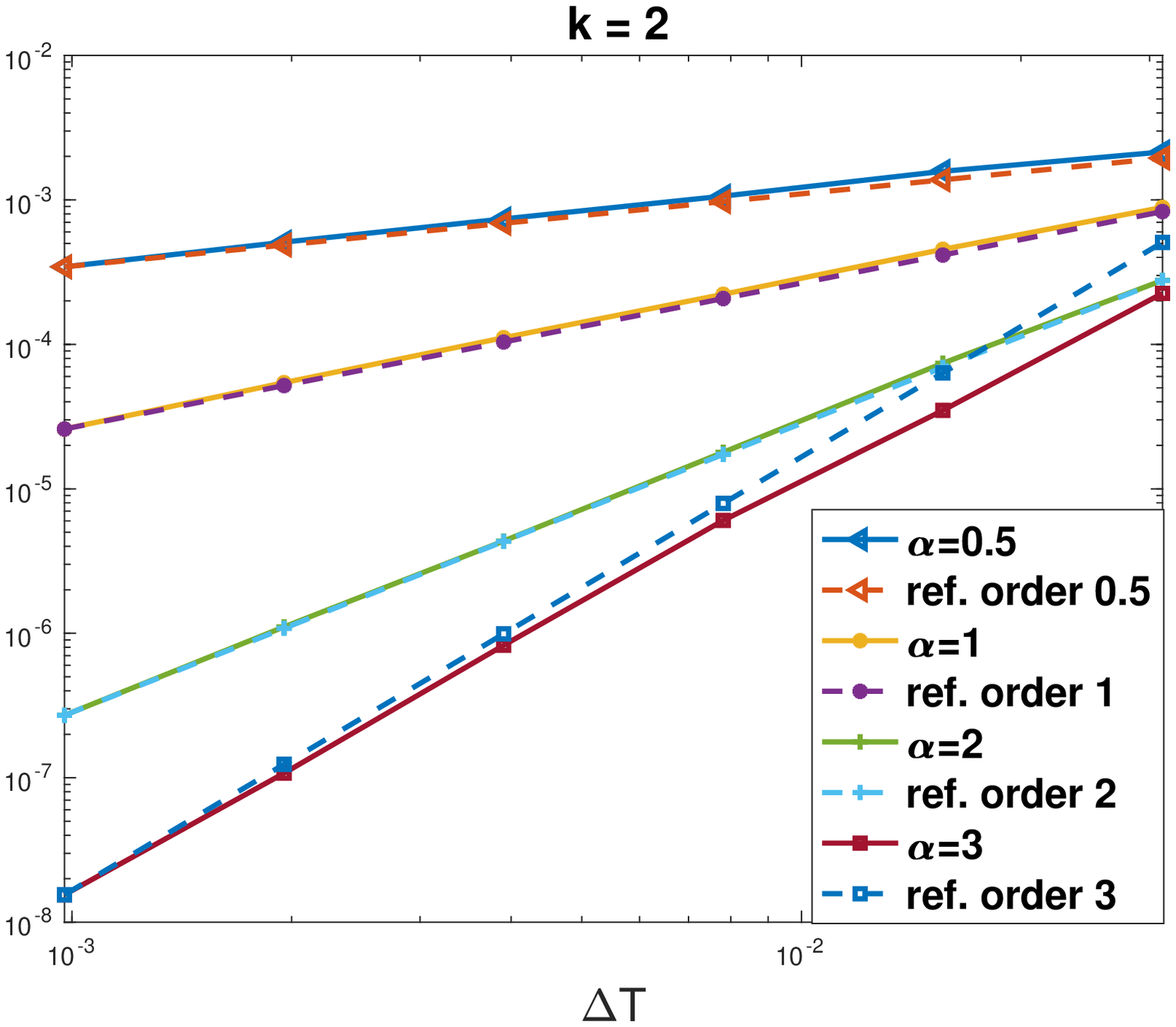}
  \includegraphics[width=7.6cm]{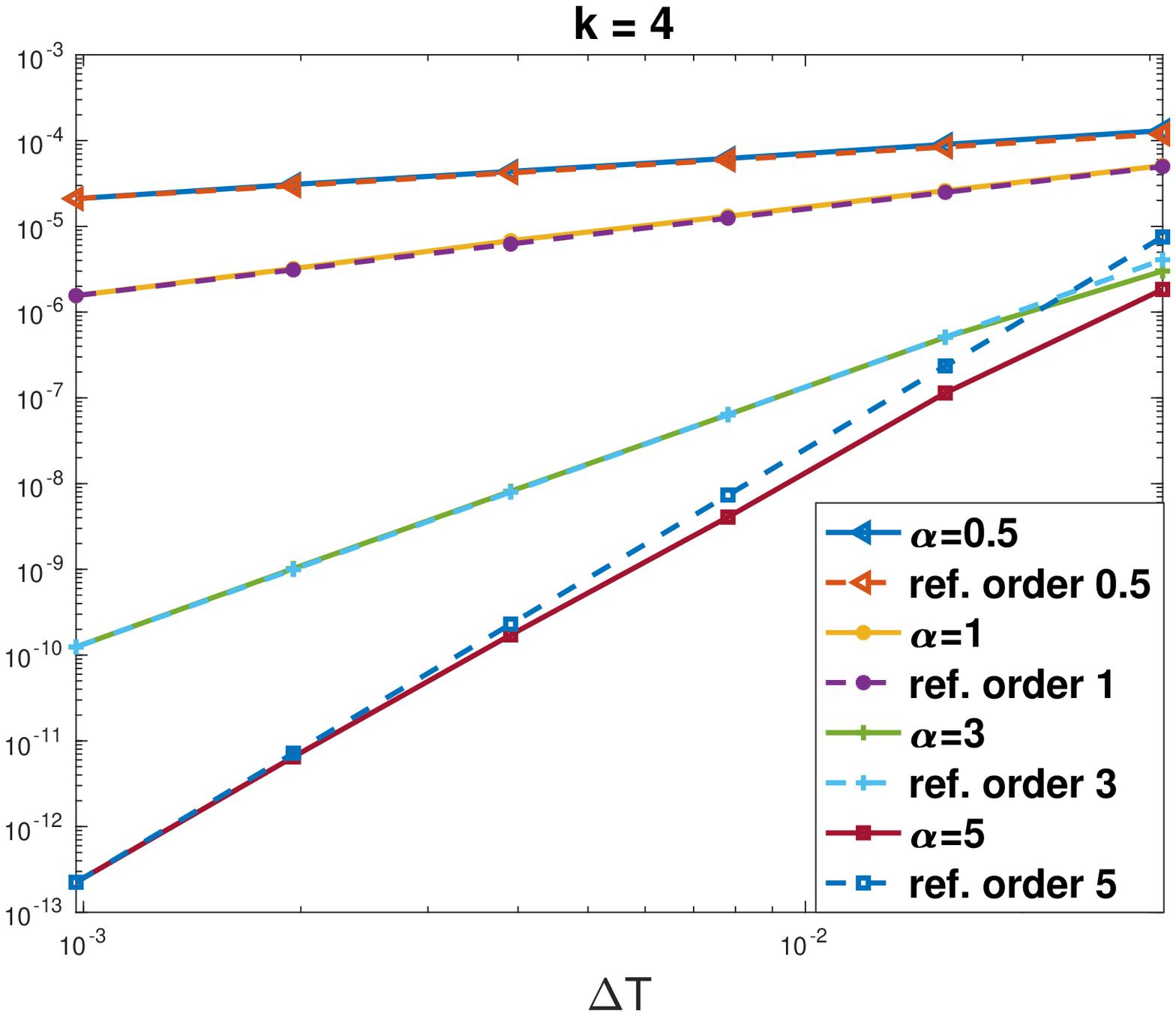}
  \caption{
Orders of convergence of the error with respect to $\Delta T$, for $k=2$ (left) and $k=4$ (right), for different values of $\al$, in the linear implicit Euler scheme case.
  }
  \label{fig:imp-2}
\end{figure}

\begin{figure}[h]
\centering
  \includegraphics[width=7.6cm]{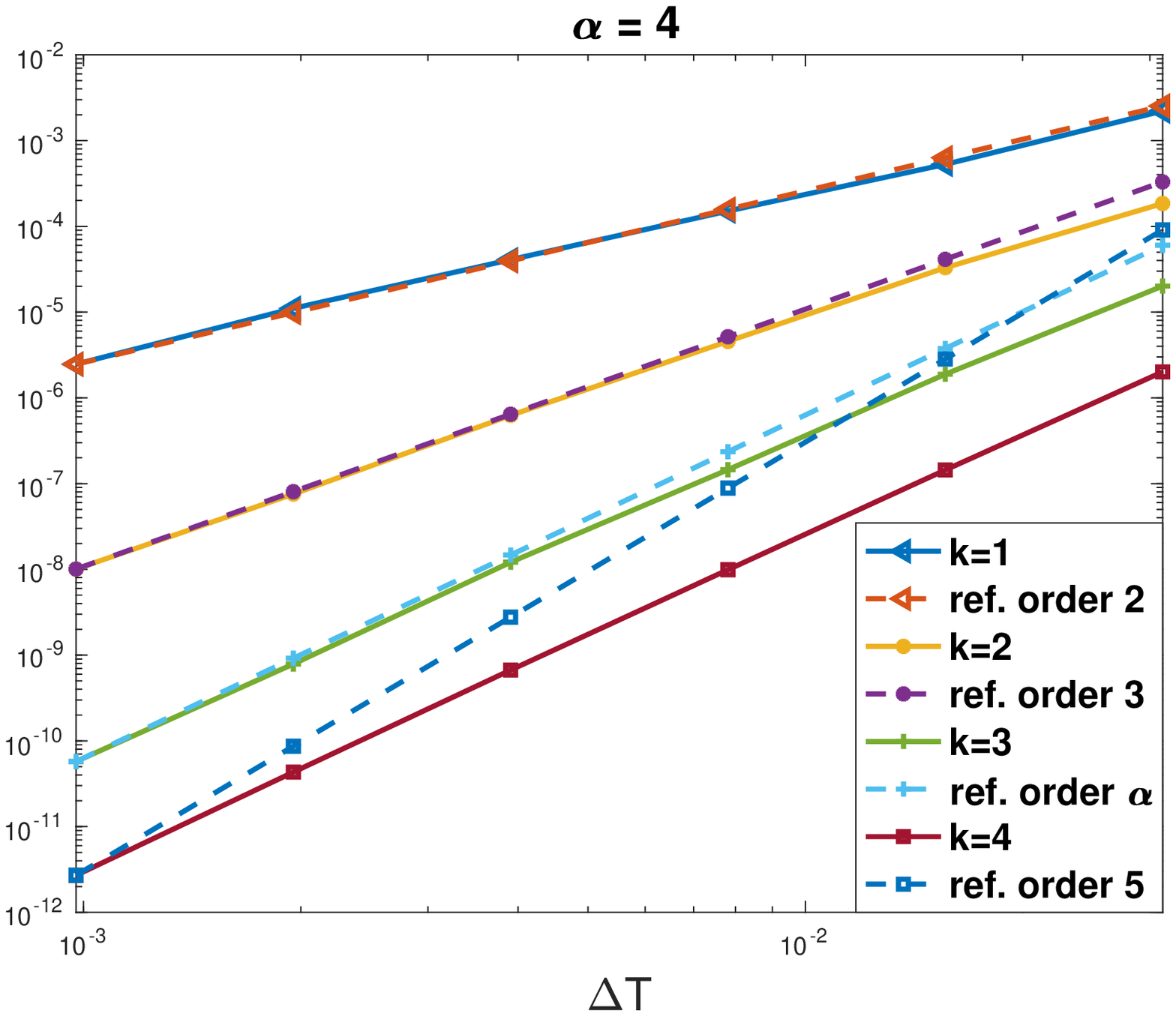}
  \includegraphics[width=7.6cm]{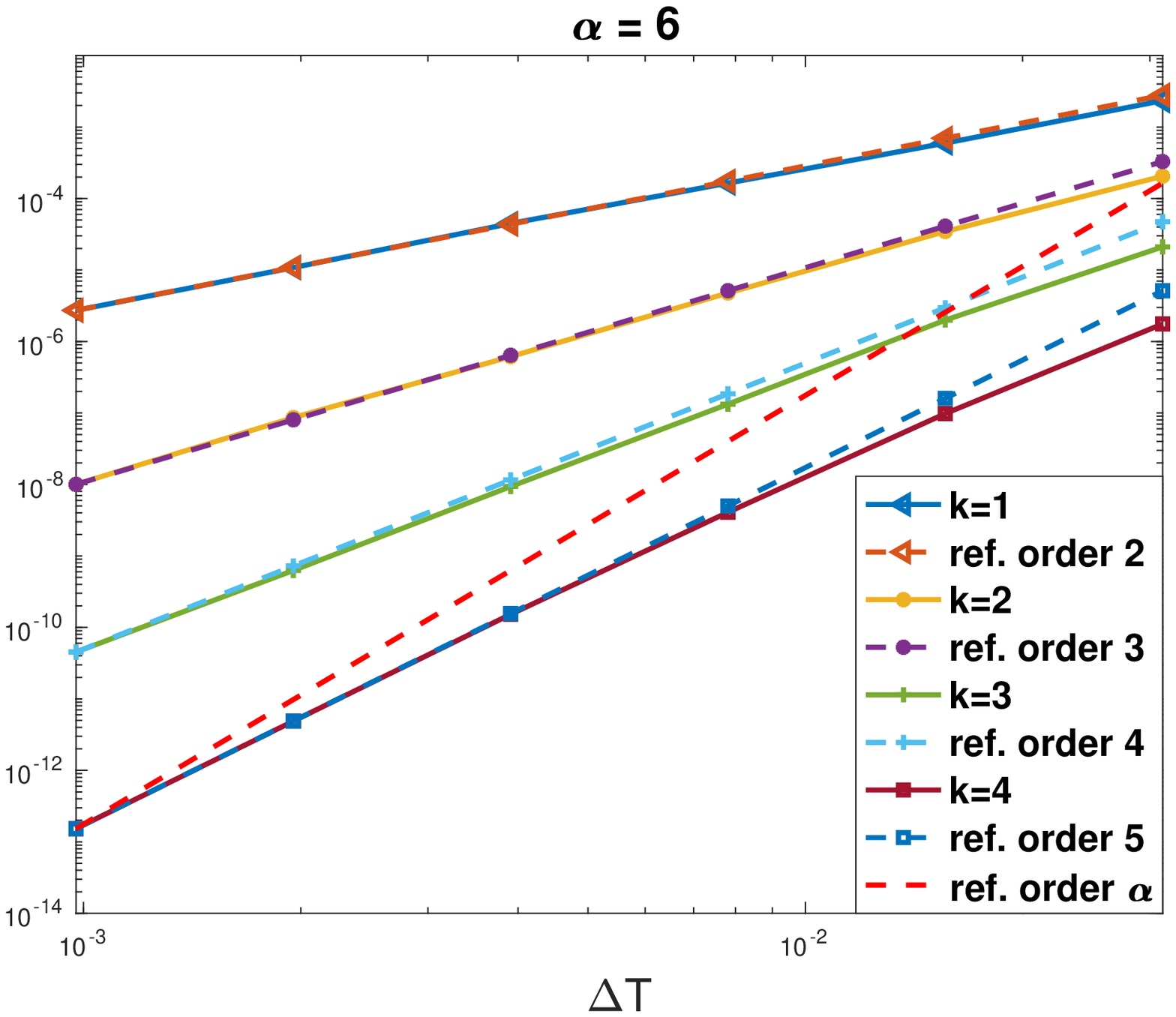}
  \caption{
  Orders of convergence of the error with respect to $\Delta T$, for $\al=4$ (left) and $\al=6$ (right), for different values of $k\in\{1,2,3,4\}$, in the linear implicit Euler scheme case.
}
  \label{fig:imp-3}
\end{figure}

To conclude this section, we report numerical simulations in the semilinear case.

Figures \ref{fig:imp4} and \ref{fig:imp5} show the order for semilinear equation \eqref{eq:SPDE} with $F(u)=\cos(u)$ and $F(u)=5\cos(u)$, respectively. The order for the additive noise case with $\al=4$ (on the left) is limited to $\frac32$ when $k\ge2$, which is the same as the deterministic case (on the right).

\begin{figure}[h]
\centering
  \includegraphics[width=7.6cm]{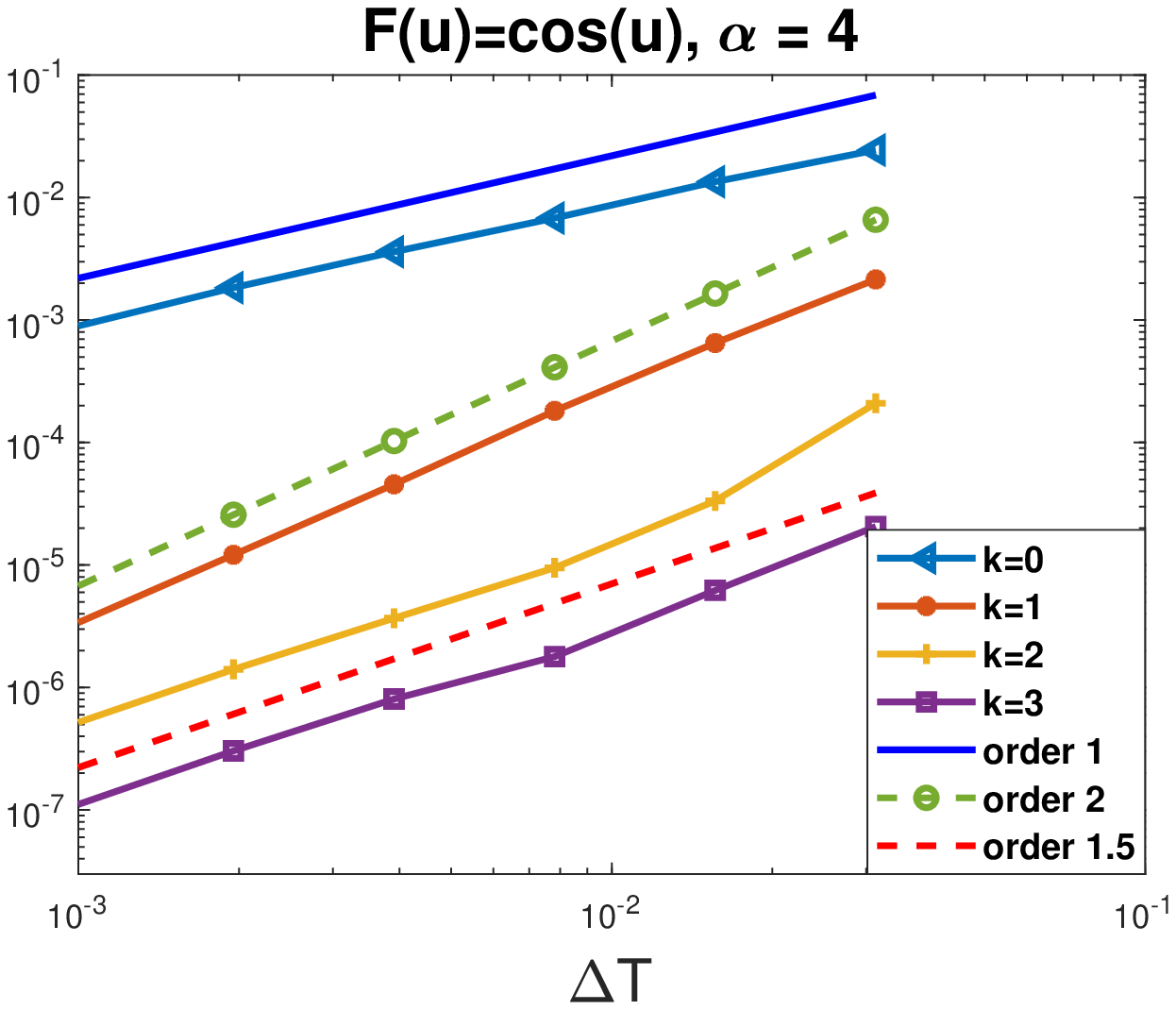}
  \includegraphics[width=7.6cm]{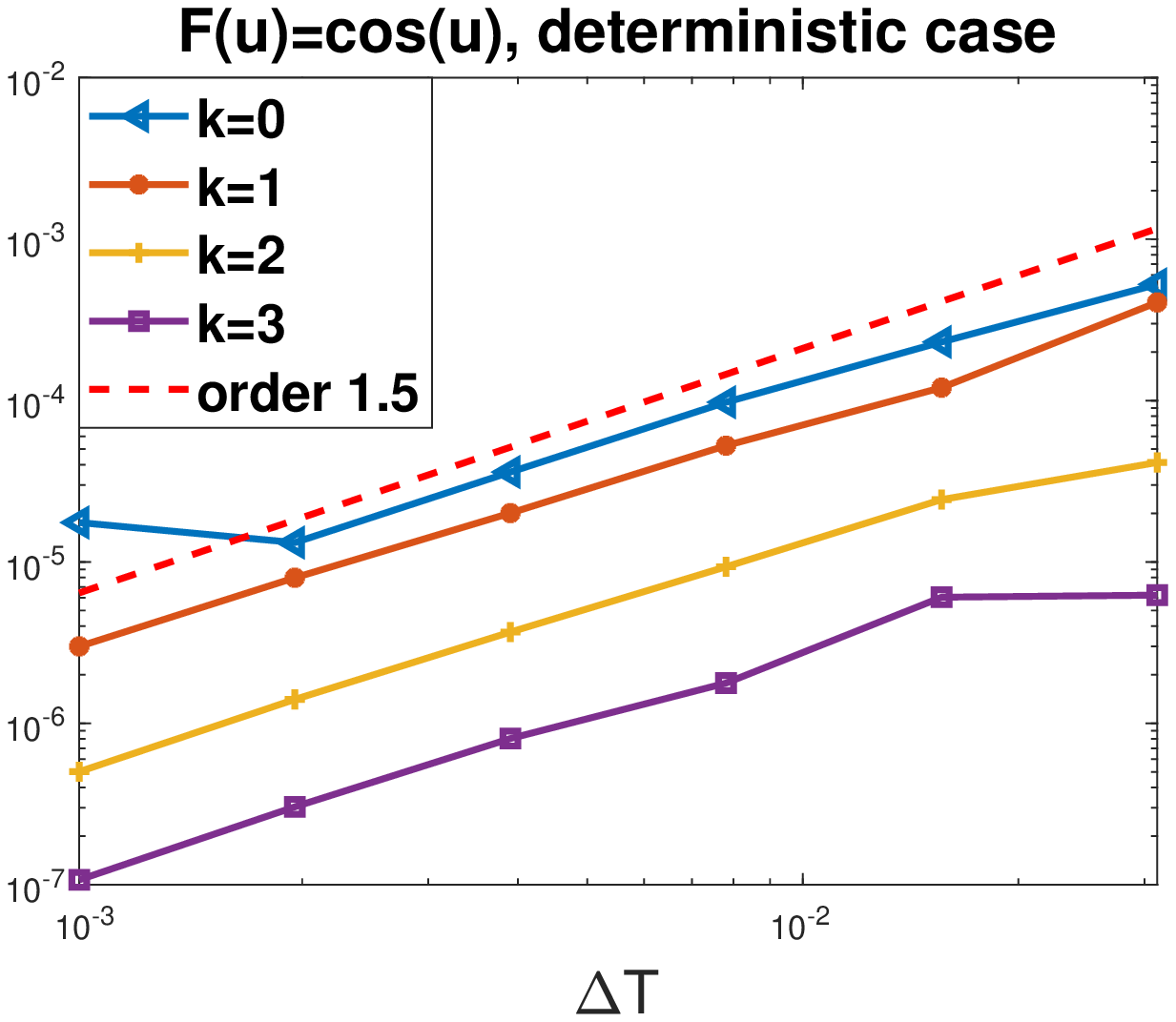}
  \caption{Orders of convergence of the error with respect to $\Delta T$, for additive noise case with $\al=4$ (left) and deterministic case (right), for nonlinear term $F(u)=\cos(u)$ and different values of $k\in\{0,1,2,3\}$, in the linear implicit Euler scheme case.}
  \label{fig:imp4}
\end{figure}

\begin{figure}[h]
\centering
  \includegraphics[width=7.6cm]{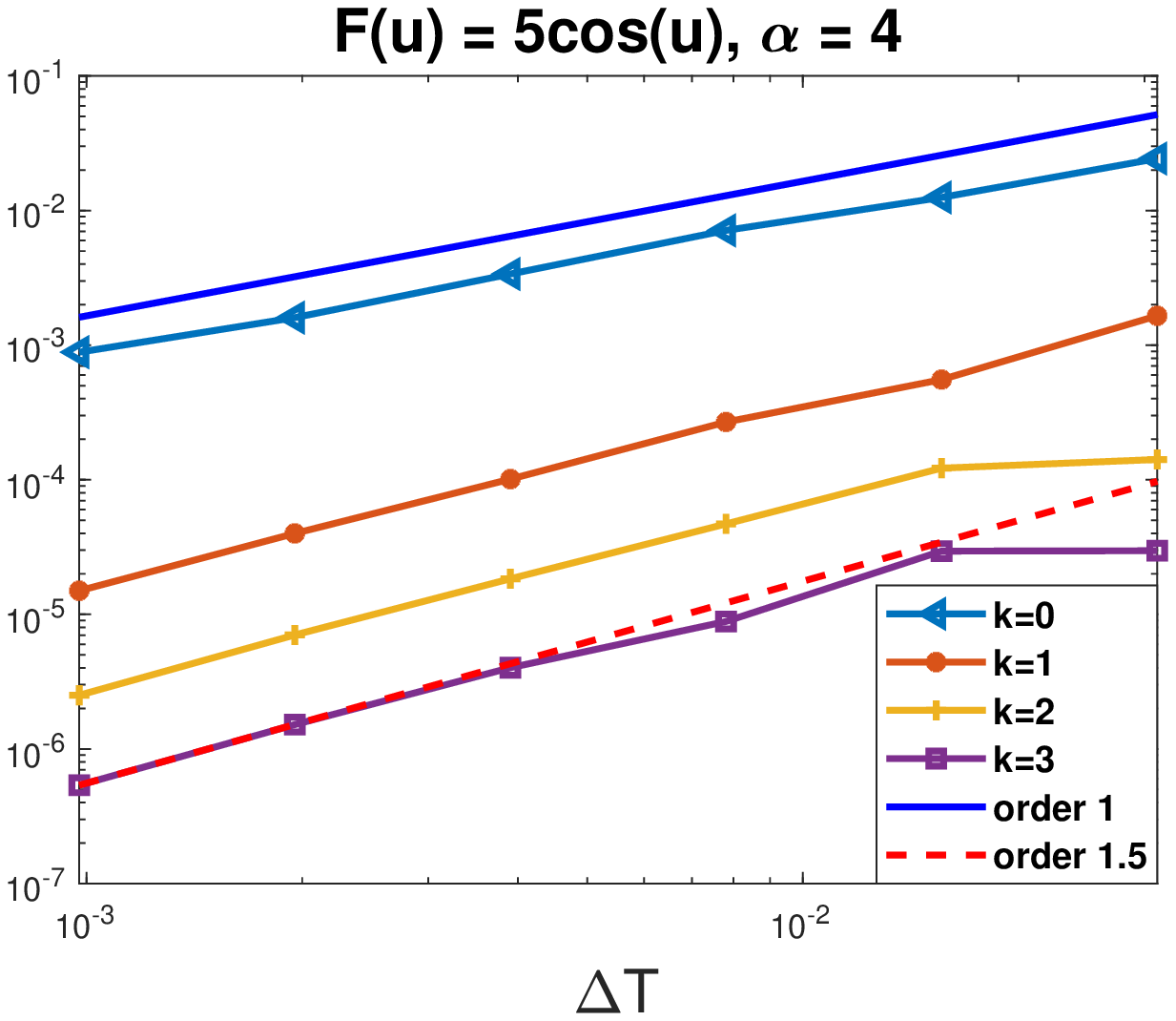}
  \includegraphics[width=7.6cm]{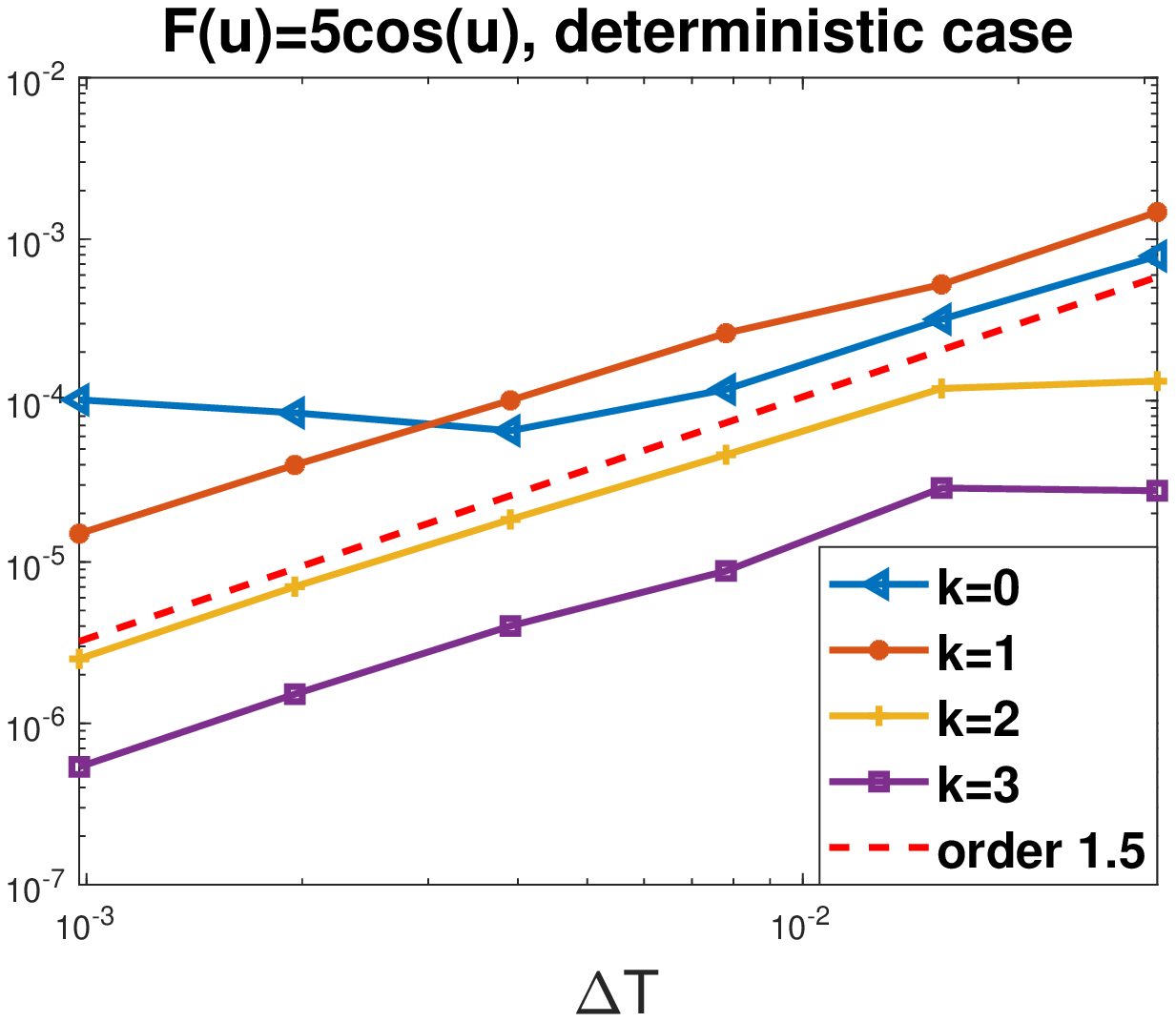}
  \caption{Orders of convergence of the error with respect to $\Delta T$, for additive noise case with $\al=4$ (left) and deterministic case (right), for nonlinear term $F(u)=5\cos(u)$ and different values of $k\in\{0,1,2,3\}$, in the linear implicit Euler scheme case.}
  \label{fig:imp5}
\end{figure}

\section{Exponential Euler scheme as the coarse integrator}\label{sec:5}

The objective of this section is to prove that, when the exponential Euler scheme is chosen as the coarse integrator, {\it i.e.} $\G=\G_{\rm expo}$ with $\hat S_{\Delta T}=e^{\Delta TA}$, then parareal iterations improve the rate of convergence of the error $\epsilon_n^{(k)}$ to $0$, in terms of $\Delta T$. Contrary to the situation of Section~\ref{sec:4}, this effect holds true without restrictions on the regularity parameter $\al$, in particular for $\al=\frac14$ (space-time white noise).

The analysis in this section is performed for the SPDE~\eqref{eq:SPDE}, with the nonlinear coefficient $F$ satisfying Assumption~\ref{ass:F}. Let also Assumption~\ref{ass:Q} be satisfied.

The content of this section is organized as follows. The main results of this sections are the error estimates stated in Theorems~\ref{theo:error-exp1} -- which gives an order of convergence $(k+1)\min(\al,\frac12)$ for all $k\in\N_0$ -- and~\ref{theo:error-exp2} -- which gives an improved order of convergence when $k\ge 2$, see Section~\ref{sec:5-error}. Numerical experiments in Section~\ref{sec:5-num} illustrate that the result in Theorem~\ref{theo:error-exp1} is sharp when $k=0$ and $k=1$, and that indeed better convergence rates are obtained for $k\ge 2$. Proofs of the results are provided in Section~\ref{sec:5-proof}, based on auxiliary results which are proved in Section~\ref{sec:5-proof-aux}.

Observe that, when the coarse integrator is the exponential Euler scheme, then the recursion formula~\eqref{eq:error} for the error yields the equality
\begin{equation}\label{eq:error_exp}
\epsilon_{n}^{(k+1)}
=\Delta T\sum_{m=0}^{n-1}e^{(n-m)\Delta T A}\left[F(u_{m}^{(k+1)})-F(u^{\rm ref}_{m})\right]+\sum_{m=0}^{n-1}e^{(n-1-m)\Delta T A}\left[\RR_m(u_{m}^{(k)})-\RR_m(u_{m}^{\rm ref})\right],
\end{equation}
where we recall that the residual operator $\RR_n$ are defined by~\eqref{eq:residual_def}. As will be clear below, if $F=0$ then $\epsilon_{n}^{(k)}=0$ for all $n$, as soon as 
$k\ge 1$. This property reveals why the choice of the exponential Euler scheme as the coarse and the fine integrator provides better results.

Note also that the fact that the noise is additive in~\eqref{eq:SPDE} is fundamental in the analysis. In addition, since noise is additive, it is expected that the order of convergence for $k=0$ may be equal to $1$ (instead of $\frac12$) if $\al$ is sufficiently large (at least larger than $\frac12$). This effect is not considered below, since the objective is mainly to study the increase in the order of convergence produced by parareal iterations, and in particular in situations where the noise is not very regular, {\it i.e.} for space-time white noise, with $\al=\frac14$.

\subsection{Statement of error estimates}\label{sec:5-error}

For the analysis, it is important to first state moment bounds for the solution $u_n^{(k)}$, which are similar to Propositions~\ref{propo:bound_u} and~\ref{propo:bound_ref} for the exact solution $u(t_n)$ and the reference solution $u_n^{\rm ref}$ respectively.
\begin{propo}\label{propo:bound_num}
For $T\in(0,\infty)$, $k\in\N_0$, $q\in\N$. There exists $C_{T,k,q}\in(0,\infty)$ such that, for all $u_0\in H$ and $\Delta T\in(0,1)$,
\[
\sup_{n\Delta T\le T}\vvvert u_n^{(k)}\vvvert_q\le C_{T,k,q}(1+|u_0|).
\]
Moreover, let $\alpha\in[0,\min(\al,\frac12))$. There exists $C_{T,k,q,\alpha}\in(0,\infty)$ such that, for all $u_0\in D((-A)^{\alpha})$ and $\Delta T\in(0,1)$,
\[
\sup_{n\Delta T\le T}\vvvert u_n^{(k)}\vvvert_{q,\alpha}\le C_{T,k,q,\alpha}(1+|u_0|_\alpha).
\]
\end{propo}

The first error estimates are stated in Theorem~\ref{theo:error-exp1}, which may be interpreted as follows: each parareal iteration improves the rate of convergence, proportionally to $\min(\al,\frac12)$.
Let us stress that Theorem~\ref{theo:error-exp1} is optimal for $k=0$ and $k=1$, as illustrated by the numerical experiments reported in Section \ref{sec:5-num}.

\begin{theo}\label{theo:error-exp1}
Let $T\in(0,\infty)$, $k\in\N_0$, $q\in\N$, $\alpha\in[0,\min(\al,\frac12))$, and arbitrarily small $\kappa>0$. For all $u_0\in D((-A)^{\alpha})$, there exists $C_{T,k,q,\alpha,\kappa}(u_0)\in(0,\infty)$ such that, for all  $\Delta T\in(0,1)$,
\[
\sup_{n\Delta T\le T}\vvvert \epsilon_n^{(k)}\vvvert_q\le C_{T,k,q,\alpha,\kappa}(u_0)\Delta T^{(k+1)(\alpha-\kappa)}.
\]
\end{theo}

Following the discussion in Section~\ref{sec:cost} concerning the computational cost, Theorem~\ref{theo:error-exp1} shows that applying the parareal algorithm may be used to reduce the computational cost, if the coarse integrator is the exponential Euler scheme, whatever the regularity of the noise.

The second result, Theorem~\ref{theo:error-exp2}, states that the estimate from Theorem~\ref{theo:error-exp2} can be improved when $k\ge 2$. The practical relevance of this result is questionable: it requires $k\ge 2$, whereas it is expected (see Section~\ref{sec:cost}) that choosing $k=1$ is optimal. Nevertheless, the study of the phenomenon stated in Theorem \ref{theo:error-exp2} is motivated by the numerical experiments reported in Section \ref{sec:5-num}, which exhibit that indeed the order of convergence is larger than $(k+1)\min(\al,\frac12)$, namely it is equal at least equal to $(k-1)\min(2\al,\frac12)+2\alpha\le k\min(2\al,\frac12)$.
\begin{theo}\label{theo:error-exp2}
Let $T\in(0,\infty)$, $k\in\N\setminus\{1\}$, $q\in\N$, $\alpha\in(0,\min(\al,\frac12))$, and arbitrarily small $\kappa>0$. For all $u_0\in D((-A)^{\alpha})$, there exists $C_{T,k,q,\alpha,\kappa}(u_0)\in(0,\infty)$ such that, for all $\Delta T\in(0,1)$,
\[
\sup_{n\Delta T\le T}\vvvert \epsilon_n^{(k)}\vvvert_{q}\le C_{T,k,q,\alpha,\kappa}(u_0)\Delta T^{(k-1)(\min(2\alpha,\frac12)-\kappa)+2\alpha}.
\]
\end{theo}

The proofs of Proposition~\ref{propo:bound_num} and Theorems~\ref{theo:error-exp1} and~\ref{theo:error-exp2} are postponed to Section~\ref{sec:5-proof}.

To simplify the exposition, the way the constants $C_{T,k,q,\alpha,\kappa}(u_0)$ above depend on $|u_0|_\alpha$ is not made precise.

\begin{rem}
Theorem~\ref{theo:error-exp1} (and Lemma~\ref{lm:LipR}--$({\rm i})$) can be proved under less restrictive conditions on the linear operator $F$, instead of Assumption~\ref{ass:F}: there exists $\eta\in(0,\frac12)$ such that one has estimates of the type
\[
|DF(u).h|_{-\eta-\alpha}\le C_{F,\eta,\alpha,\kappa}\bigl(1+|u|_{\alpha+\kappa}\bigr)|h|_{-\alpha},
\]
and
\[
|D^2F(u).(h_1,h_2)|_{-\eta}\le C_{F,\eta}|h_1||h_2|.
\]
This setting encompasses the case of Nemytskii operators, with $\eta\in(\frac14,\frac12)$.
\end{rem}

\subsection{Numerical experiments}\label{sec:5-num}

The objective of this section is to illustrate Theorem~\ref{theo:error-exp1} and~\ref{theo:error-exp2}.

The SPDE~\eqref{eq:SPDE}, with nonlinear operator $F(u)=5\cos(u)$ and initial condition $u(0)=0$, is considered, with covariance operator given by $Qe_p=\gamma_pe_p$, with $\gamma_p=\lambda_p^{\frac12-2\al}$. Spatial discretization is performed using finite differences, with mesh size $h=0.01$. In addition, the noise is truncated, {\it i.e.} the $Q$-Wiener process $W^Q(t)$ is replaced by $\sum_{p=1}^{P}\gamma_p^{\frac12}\beta_p(t)e_p$, with $P=100$. Numerical parameters are chosen as follows: the final time is $T=1$, the fine time-step size is $\delta t=2^{-15}$, and the coarse time-step size is $\Delta T=J\delta t$ with $J=2^{j}$, $j=5,\cdots,10$. An average over $M=100$ independent Monte-Carlo samples is used to approximate the expectations.

\begin{figure}[h]
\centering
  \includegraphics[width=7.6cm]{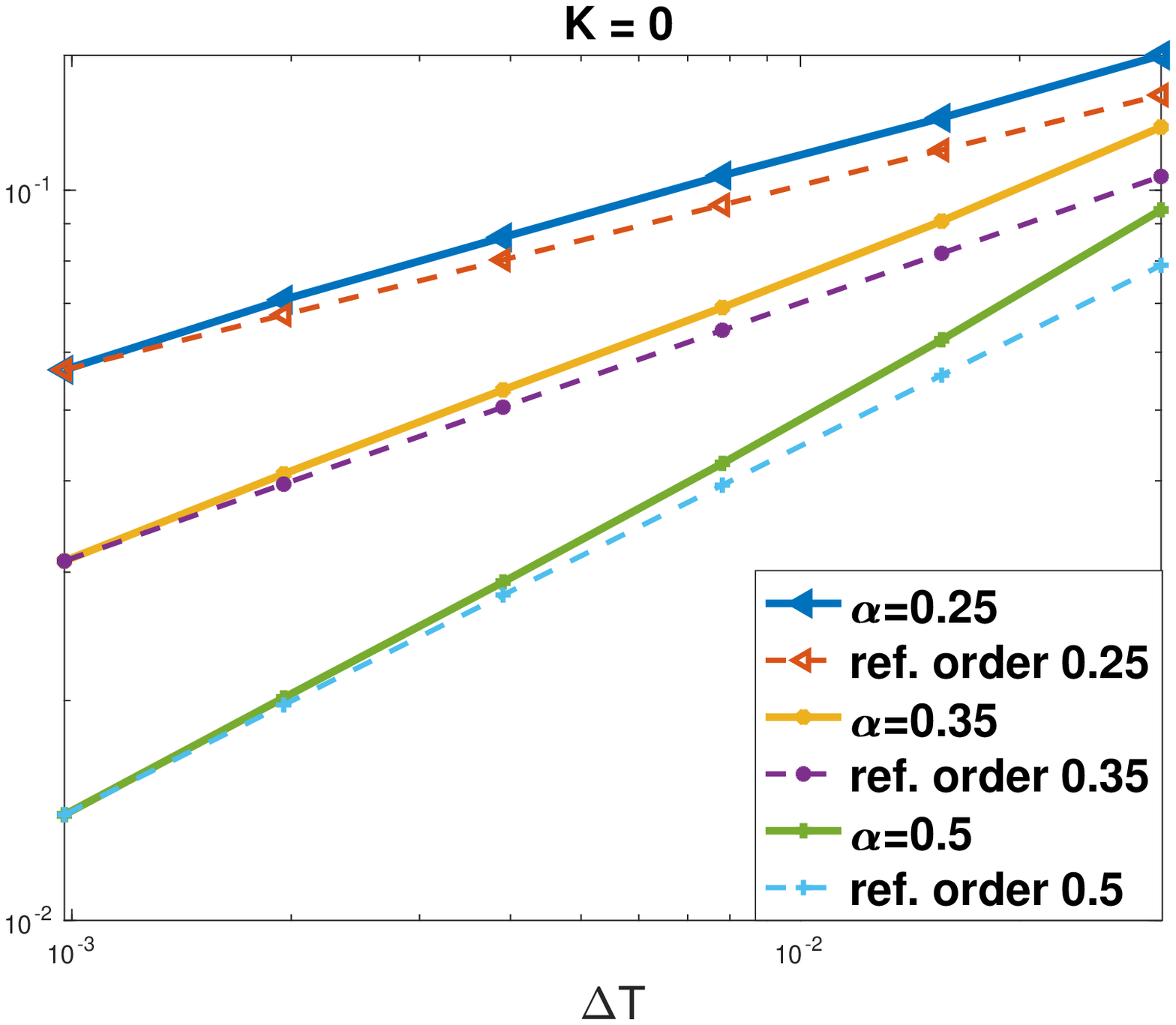}
  \includegraphics[width=7.6cm]{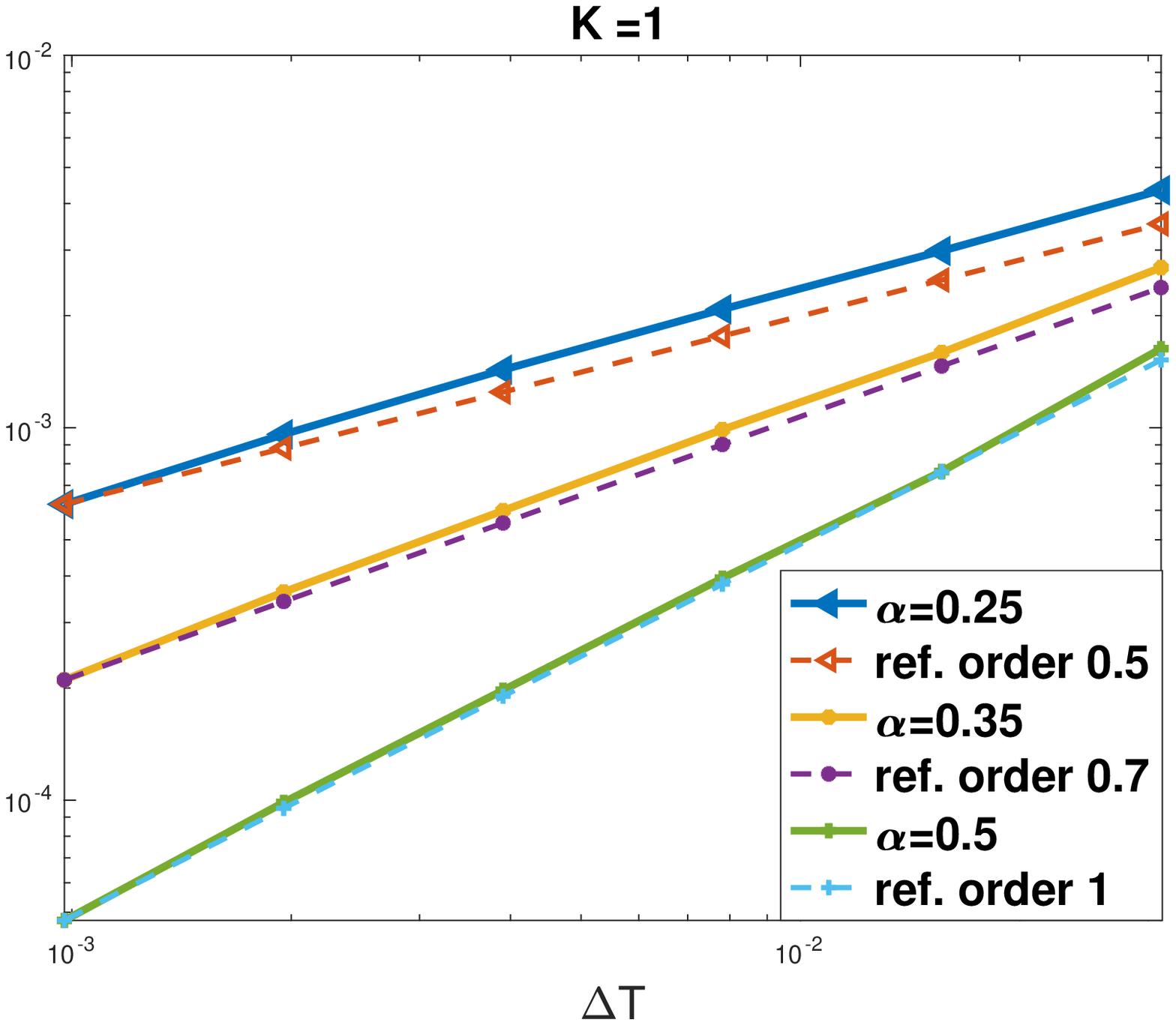}
  \caption{
Orders of convergence of the error with respect to $\Delta T$, for $K=0$ (left) and $K=1$ (right), for different values of $\alpha\in\{0.25,0.35,0.5\}$, in the exponential Euler scheme case.
}
  \label{fig:k01}
\end{figure}

\begin{figure}[h]
\centering
  \includegraphics[width=7.6cm]{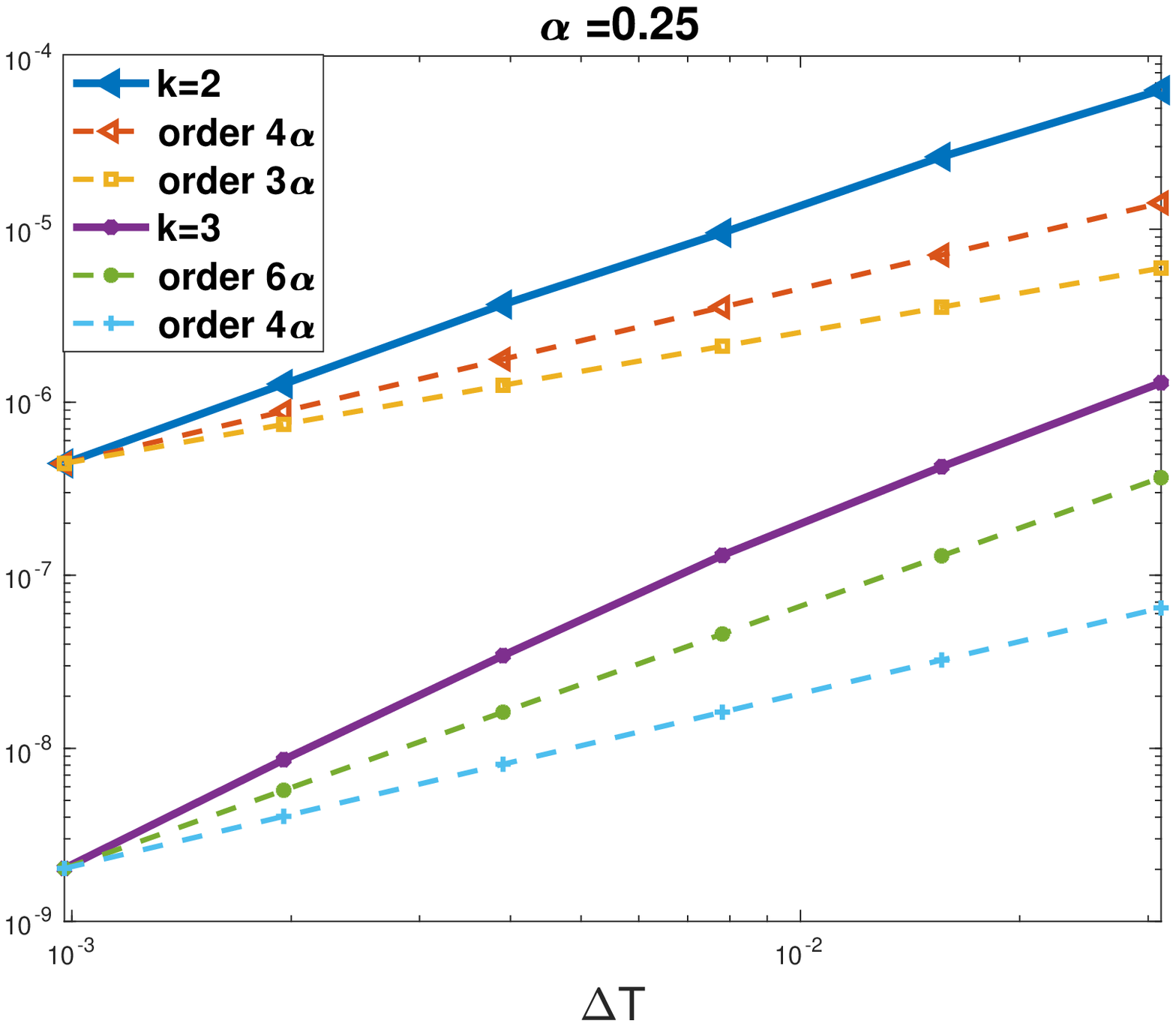}
  \includegraphics[width=7.6cm]{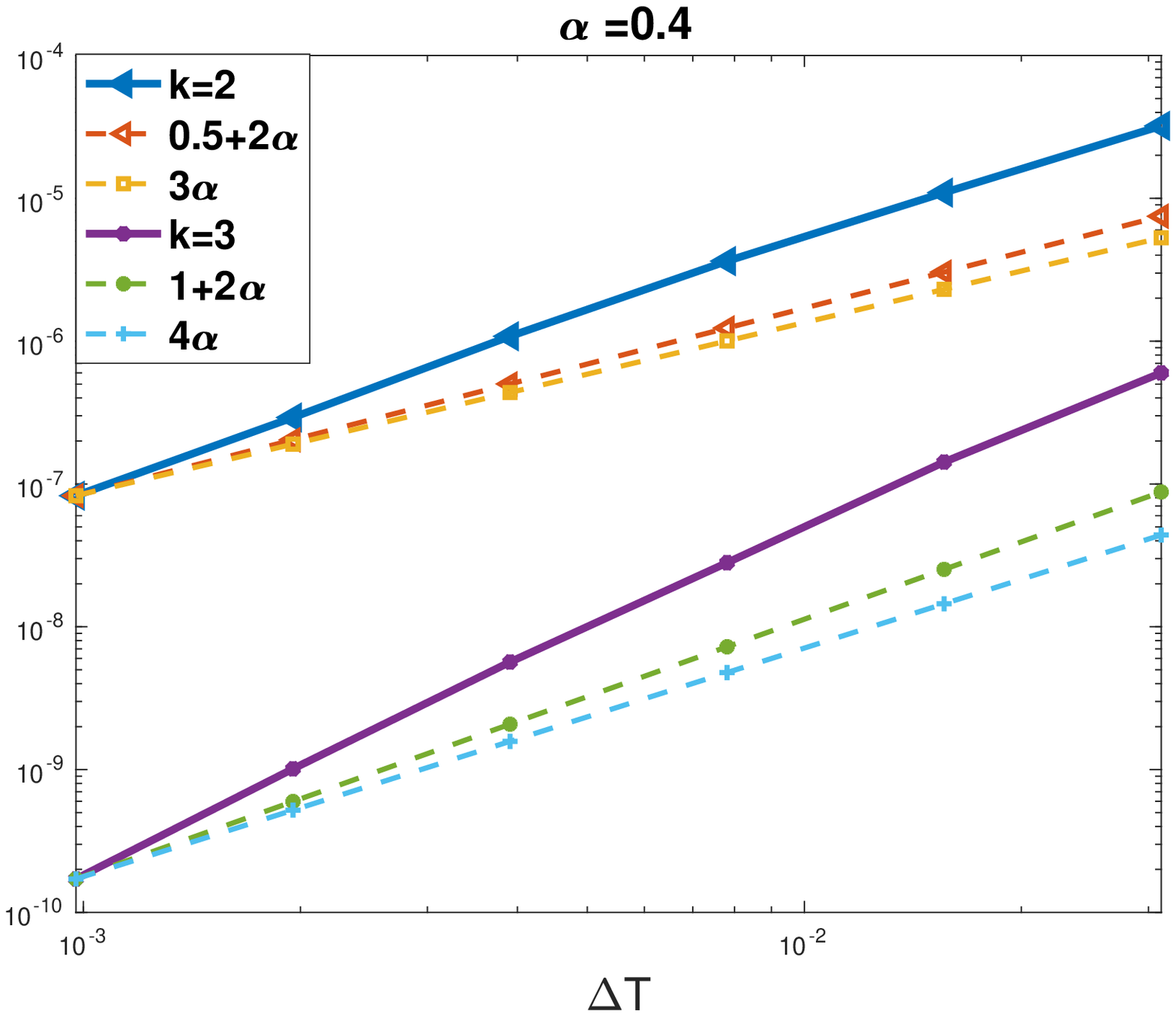}
  \caption{
Orders of convergence of the error with respect to $\Delta T$, for $\alpha=0.25$ (left) and $\alpha=0.4$ (right), for different values of $k\in\{2,3,4\}$, in the exponential Euler scheme case.
}
  \label{fig:k23}
\end{figure}

Figure~\ref{fig:k01} demonstrates that the orders of convergence in Theorem~\ref{theo:error-exp1} are sharp when $k=0$ and $k=1$, for different values of $\al$.

Figure~\ref{fig:k23} then illustrates that for $k\ge 2$, the orders of convergence in Theorem~\ref{theo:error-exp1} are not sharp for $\al=0.25$ and $\al=0.4$. In fact, even the improved error estimates from Theorem~\ref{theo:error-exp2} seem to be sub-optimal in the numerical experiment, especially when $\al=0.4$ (figure on the right) -- in this case $\min(2\alpha,\frac12)=\frac12$. When $\al=0.25$ (figure on the left), the improved theoretical rate in Theorem \ref{theo:error-exp2} is closer to the observed numerical rate.

\subsection{Proof of error estimates}\label{sec:5-proof}

The objective of this section is to provide proofs for the results stated in Section~\ref{sec:5-error}. In fact, they are all based on the following Lemma concerning Lipschitz continuity properties, in appropriate norms, of the residual operators.

\begin{lemma}\label{lm:LipR}
Let Assumptions \ref{ass:F} and \ref{ass:Q} hold, and $\alpha\in[0,1)$.

$({\rm i})$ There exists $C_{F,\alpha}\in(0,\infty)$ such that for all $\Delta T\in(0,1]$ and all $u_1,u_2\in H$, one has
\[
\underset{n\Delta T\le T}{\sup}|\RR_n(u_2)-\RR_n(u_1)|_\alpha\le C_{F,q,\alpha} \Delta T^{1-\alpha}|u_2-u_1|.
\]

$({\rm ii})$ For all $q\in\N$, $\alpha\in(0,\min(\al,\frac12))$, $\beta\in[0,\alpha)$ and arbitrarily small $\kappa\in\bigl(0,\min(\alpha-\beta,1-\alpha-\beta)\bigr)$, there exists $C_{T,q,\alpha,\beta,\kappa}\in(0,\infty)$ such that for all $\Delta T\in(0,1)$ and all $u_1,u_2\in D((-A)^{\alpha})$, one has
\begin{equation*}
\vvvert(-A)^{-\alpha}\bigl(\RR_n(u_2)-\RR_n(u_1)\bigr)\vvvert_q\le C_{T,q,\alpha,\beta,\kappa}\Delta T^{1+\min\bigl(\alpha+\beta,\frac12\bigr)-\kappa}\bigl(1+|u_1|_{\alpha}^2+|u_2|_{\alpha}^2\bigr)|u_2-u_1|_{\beta}.
\end{equation*}

\end{lemma}

The proof of Lemma~\ref{lm:LipR} is postponed to Section~\ref{sec:5-proof-aux}.

Observe that Statement $({\rm i})$ (in the case $\alpha=0$) is not sufficient to exhibit a positive rate of convergence $\gamma$ for the error $\epsilon_{n}^{(k)}$, in terms of $\Delta T$. Indeed, that would required the Lipschitz constant of the residual operator to be of size $\Delta T^{1+\gamma}$. Statement $({\rm i})$, with $\alpha=0$, is used only to establish stability properties of the algorithm, {\it i.e.} moment bounds for $u_n^{(k)}$ or to apply a Gronwall lemma. With $\alpha>0$, this statement is employed to establish moment bounds in $|\cdot|_{\alpha}$ norms, and to deal with some terms for which regularization properties of the semigroup cannot be used.

Statement $({\rm ii})$ is the key result in order to get the rates of convergence given in Theorem~\ref{theo:error-exp1} and~\ref{theo:error-exp2}. It requires to choose appropriate norms to analyze the Lipschitz constant of the residual operators: estimation in a weaker norm for $\alpha>0$, or with higher regularity for $\beta>0$. More precisely, Theorem~\ref{theo:error-exp1} follows from Statement $({\rm ii})$ with $\beta=0$, and then Theorem~\ref{theo:error-exp2} follows from Statement $({\rm ii})$ with $\beta>0$.

Finally, observe that the need to choose appropriate norms is specific to the infinite dimensional situation. Moreover, Lemma~\ref{lm:LipR} is not satisfied if the linear implicit Euler scheme is chosen as the coarse integrator, even if $F=0$: indeed, in that case, $\RR_n(u_2)-\RR_n(u_1)=\bigl(S_{\Delta T}-e^{\Delta T}\bigr)(u_2-u_1)$, and $\|S_{\Delta T}-e^{\Delta T}\|_{\mathcal{L}(H)}$ does not converge to $0$ as $\Delta T\to 0$. This observation explains why the behaviors and the analysis of the parareal algorithm depends a lot on the choice of the coarse integrator for parabolic semilinear SPDEs.

\begin{rem}\label{rem:condition}
Note that the estimates stated in Lemma~\ref{lm:LipR} require $u_1$ and $u_2$ to be deterministic, however below they are applied to random elements, which are measurable with respect to $\sigma\bigl(W(t);t\le t_n\bigr)$, whereas the noise component in the residual operator $\RR_n$ is measurable with respect to $\sigma\bigl(W(t); t_n\le t\le t_{n+1}\bigr)$. Applying a straightforward conditioning argument, and interpreting the expectation in Lemma~\ref{lm:LipR} as a conditional expectation, yield the required estimates below.
\end{rem}

\begin{proof}[Proof of Proposition~\ref{propo:bound_num}]
First, moment bounds for the reference solution $u_n^{\rm ref}$ are provided by Proposition~\ref{propo:bound_ref}. In addition, in this section, the initialization step of the parareal algorithm consists in applying the fine integrator (which is the exponential Euler scheme) with time-step size $\Delta T$, thus the result of Proposition~\ref{propo:bound_ref} also applies to obtain moment bounds for $u_n^{(0)}$.

Since $u_n^{(k)}=\epsilon_n^{(k)}+u_{n}^{\rm ref}$, it thus only remains to prove moment bounds for $\epsilon_{n}^{(k)}$, when $k\ge 1$.

Using the expression of the error~\eqref{eq:error_exp}, and the global Lipschitz continuity of $F$, moments are treated as follows: for all $q\in\N$ and $\alpha\in[0,1)$,
\begin{align*}
\vvvert \epsilon_n^{(k+1)}\vvvert_{q,\alpha}
\le& \Delta T\sum_{m=0}^{n-1}\|(-A)^{\alpha}e^{(n-m)\Delta T A}\|_{\mathcal{L}(H)}\vvvert \epsilon_m^{(k+1)}\vvvert_{q,0}\\ 
&+\sum_{m=0}^{n-2}\|(-A)^{\alpha}e^{(n-1-m)\Delta T A}\|_{\mathcal{L}(H)}\vvvert \RR_m(u_m^{(k)})-\RR_m(u_m^{\rm ref})\vvvert_{q,0}\\
&+\vvvert (-A)^{\alpha}\bigl[\RR_{n-1}(u_{n-1}^{(k)})-\RR_{n-1}(u_{n-1}^{\rm ref})\bigr]\vvvert_{q,0}.
\end{align*}
Assume first that $\alpha=0$. Then the claim follows from the application of Lemma~\ref{lm:LipR}--$({\rm i})$, the use of the discrete Gronwall Lemma, and from the use of a recursion argument with respect to $k$. When $\alpha>0$, it remains to apply Lemma~\ref{lm:LipR}--$({\rm i})$, and to apply the regularization estimate from Proposition~\ref{propo:A1} to conclude the proof.
\end{proof}

\begin{proof}[Proof of Theorem~\ref{theo:error-exp1}]

First, assume that $k=0$. The claim then follows from Proposition~\ref{propo:bound_ref}, since both the coarse and the fine integrators are given by the exponential Euler scheme, thus $\epsilon_n^{(k)}=\bigl(u_n^{(k)}-u(t_n)\bigr)-\bigl(u_n^{\rm ref}-u(t_n)\bigr)$, and the initialization of the parareal algorithm consists in applying the coarse integrator.

Let now $k\ge 1$. Owing to the expression of the error~\eqref{eq:error_exp}, applying Lemma~\ref{lm:LipR}--$({\rm ii})$ with $\beta=0$ and H\"older inequality (see Remark~\ref{rem:condition} for the conditioning argument), one obtains
\begin{align*}
\vvvert \epsilon_n^{(k)}\vvvert_q
\le& C\Delta T\sum_{m=0}^{n-1}\vvvert \epsilon_m^{(k)}\vvvert_{q}\\
&+\sum_{m=0}^{n-2}\|(-A)^{\alpha}e^{(n-1-m)\Delta TA}\|_{\mathcal{L}(H)}\vvvert(-A)^{-\alpha}[\RR(u_m^{(k-1)})-\RR(u_m^{\rm ref})]\vvvert_q\\
&+\vvvert \RR_{n-1}(u_{n-1}^{(k-1)})-\RR_{n-1}(u_{n-1}^{\rm ref})\vvvert_q\\
&\le C\Delta T\sum_{m=0}^{n-1}\vvvert \epsilon_m^{(k)}\vvvert_{q}+C\Delta T\vvvert \epsilon_{n-1}^{(k-1)}\vvvert_q\\
&+C\Delta T^{1+\alpha-\kappa}\sum_{m=0}^{n-2}\frac{1}{((n-m-1)\Delta T)^{\alpha}}\vvvert \epsilon_m^{(k-1)}\vvvert_{2q}(1+\vvvert u_m^{(k-1)}\vvvert_{4q,\alpha}^2+\vvvert u_m^{\rm ref}\vvvert_{4q,\alpha}^2).
\end{align*}
Due to the moment estimates of Proposition~\ref{propo:bound_num}, applying the discrete Gronwall lemma yields
\[
\sup_{n\Delta T\le T}\vvvert \epsilon_n^{(k)}\vvvert_q\le C_{T,q,\alpha,\kappa}(u_0)\Delta T^{\alpha-\kappa}\sup_{n\Delta T\le T}\vvvert \epsilon_n^{(k-1)}\vvvert_{2q}.
\]
Then a straightforward recursion argument yields
\begin{align*}
\sup_{n\Delta T\le T}\vvvert \epsilon_n^{(k)}\vvvert_q&\le C_{T,k,q,\alpha,\kappa}(u_0)\Delta T^{k(\alpha-\kappa)}\sup_{n\Delta T\le T}\vvvert \epsilon_n^{(0)}\vvvert_{2^kq}\\
&\le C_{T,k,q,\alpha,\kappa}(u_0)\Delta T^{(k+1)(\alpha-\kappa)},
\end{align*}
owiing to the result when $k=0$. This concludes the proof of Theorem~\ref{theo:error-exp1}.
\end{proof}

\begin{proof}[Proof of Theorem~\ref{theo:error-exp2}]
The proof consists of two steps.

{\bf Step~$1$.} If $k\ge 1$, the error evaluated in the $|\cdot|_\alpha$ norm is also of the order $(k+1)(\alpha-\kappa)$:
\[
\sup_{n\Delta T\le T}\vvvert \epsilon_n^{(k)}\vvvert_{q,\alpha}\le C_{T,k,q,\alpha,\kappa}(u_0)\Delta T^{(k+1)(\alpha-\kappa)}.
\]

Indeed, using the expression of the error~\eqref{eq:error_exp}, Proposition~\ref{propo:A1}, and the two statements of Lemma~\ref{lm:LipR}, one obtains
\begin{align*}
\vvvert \epsilon_n^{(k)}\vvvert_{q,\alpha}
\le &\Delta T\sum_{m=0}^{n-1}\|(-A)^{\alpha}e^{(n-m)\Delta T A}\|_{\mathcal{L}(H)}\vvvert \epsilon_m^{(k)}\vvvert_{q}\\ 
&+\sum_{m=0}^{n-2}\|(-A)^{2\alpha}e^{(n-1-m)\Delta T A}\|_{\mathcal{L}(H)}\vvvert(-A)^{-\alpha}[\RR_m(u_m^{(k-1)})-\RR_m(u_m^{\rm ref})]\vvvert_{q}\\
&+\vvvert (-A)^{\alpha}\bigl[\RR_{n-1}(u_{n-1}^{(k-1)})-\RR_{n-1}(u_{n-1}^{\rm ref})\bigr]\vvvert_{q}\\
& \le C\left(\Delta T\sum_{m=0}^{n-1}\frac{1}{((n-m)\Delta T)^\alpha}\right)\sup_{m\le n}\vvvert \epsilon_m^{(k)}\vvvert_{q}\\
&+C\left(\Delta T\sum_{m=0}^{n-1}\frac{1}{((n-1-m)\Delta T)^{2\alpha}}\right)\Delta T^{\alpha-\kappa}\sup_{m\le n}\vvvert \epsilon_m^{(k-1)}\vvvert_{q}\\
&+\Delta T^{1-\alpha}\vvvert \epsilon_{n}^{(k-1)}\vvvert_{q}\\
&\le C\Delta T^{(k+1)(\alpha-\kappa)}+C\Delta T^{1-\alpha+k(\alpha-\kappa)},
\end{align*}
owing to Theorem~\ref{theo:error-exp1}. With the assumptions $\alpha\le \frac12$, it is straightforward to check that $1-\alpha+k(\alpha-\kappa)\ge (k+1)(\alpha-\kappa)$. This concludes Step~$1$.

{\bf Step $2$.} It remains to establish the error estimate of Theorem~\ref{theo:error-exp2}.

Using the expression of the error~\eqref{eq:error_exp}, Proposition~\ref{propo:A1}, the two statements in Lemma~\ref{lm:LipR}, with $\beta=\alpha-\kappa$, and the moment bounds from Proposition~\ref{propo:bound_num}, one obtains
\begin{align*}
\vvvert \epsilon_n^{(k)}\vvvert_{q,\alpha}
\le& \Delta T\sum_{m=0}^{n-1}\|(-A)^{\alpha}e^{(n-m)\Delta T A}\|_{\mathcal{L}(H)}\vvvert \epsilon_m^{(k)}\vvvert_{q}\\ 
&+\sum_{m=0}^{n-2}\|(-A)^{2\alpha}e^{(n-1-m)\Delta T A}\|_{\mathcal{L}(H)}\vvvert (-A)^{-\alpha}\bigl[\RR_m(u_m^{(k-1)})-\RR_m(u_m^{\rm ref})\bigr]\vvvert_{q}\\
&+\vvvert \RR_{n-1}(u_{n-1}^{(k-1)})-\RR_{n-1}(u_{n-1}^{\rm ref})\vvvert_{q,\alpha}\\
\le& C\Delta T\sum_{m=0}^{n-1}\frac{1}{((n-m)\Delta T)^\alpha}\vvvert \epsilon_m^{(k)}\vvvert_{q,\alpha}\\
&+C\Delta T^{\min(2\alpha,\frac12)-2\kappa}\left(\Delta T\sum_{m=0}^{n-2}\frac{1}{((n-1-m)\Delta T)^{2\alpha}}\right)\sup_{m\le n-2}\vvvert \epsilon_m^{(k-1)}\vvvert_{2q,\alpha}\\
&+C\Delta T^{1-\alpha}\vvvert\epsilon_{n-1}^{(k-1)}\vvvert_{q}\\
\le& C\Delta T\sum_{m=0}^{n-1}\frac{1}{((n-m)\Delta T)^\alpha}\vvvert \epsilon_m^{(k)}\vvvert_{q,\alpha}+C\Delta T^{\min(2\alpha,\frac12)-2\kappa}\sup_{m\Delta T\le T}\vvvert \epsilon_m^{(k-1)}\vvvert_{2q,\alpha},
\end{align*}
using the fact that $1-\alpha\ge \frac12\ge \min(2\alpha,\frac12)-2\kappa$.

Applying the discrete Gronwall lemma and a recursion argument,
\begin{align*}
\vvvert \epsilon_n^{(k)}\vvvert_{q,\alpha}&\le C\Delta T^{(k-1)(\min(2\alpha,\frac12)-\kappa)}\vvvert \epsilon_n^{(1)}\vvvert_{2^kq,\alpha}\\
&\le C\Delta T^{(k-1)(\min(2\alpha,\frac12)-\kappa)+2\alpha},
\end{align*}
owing to Step~$1$ with $k=1$.

This concludes the proof of Theorem~\ref{theo:error-exp2}.
\end{proof}

\begin{rem}
The claim in Step $1$ in the proof of Theorem~\ref{theo:error-exp2} requires $k\ge 1$, and is not correct when $k=0$. In fact, when $k\ge 1$, in the expression of the error, contributions of stochastic terms, which would have low regularity properties, do not appear explicitly since noise is additive.

The application of the claim of Step $2$ explains why Theorem~\ref{theo:error-exp2} holds true only if $k\ge 2$.
\end{rem}

\subsection{Proof of Lemma~\ref{lm:LipR}}\label{sec:5-proof-aux}

The proof of Lemma\ref{lm:LipR} consists in proving bounds on the derivative $D\RR_n(u).h=D\F_n(u).h-DG_n(u).h$ (recall the residual operator $\RR_n$ is defined by~\eqref{eq:residual_def}). Let us introduce notation which is used below. On the one hand, due to~\eqref{eq:coarse}, it is straightforward to compute
\[
D\G_n(u).h=e^{\Delta TA}h+\Delta Te^{\Delta TA}DF(u).h.
\]
On the other hand, the derivative of $\F_n$ is computed as follows: one has $D\F_n(u).h=\eta_{n,J}^h$, where $J\delta t=\Delta T$, and one has the following recursion formulae
\begin{align*}
v_{n,j+1}&=e^{\delta tA}v_{n,j}+\delta te^{\delta tA}F(v_{n,j})+e^{\delta tA}\delta_{n,j}W^Q~,\quad v_{n,0}=u,\\
\eta_{n,j+1}^h&=e^{\delta tA}\eta_{n,j}^h+\delta te^{\delta tA}DF(v_{n,j}).\eta_{n,j}^h~,\quad \eta_{n,0}^h=h.
\end{align*}
Then a straightforward computation yields
\[
\eta_{n,J}^h=e^{\Delta TA}h+\delta t\sum_{j=0}^{J-1}e^{(J-j)\delta tA}DF(v_{n,j}).\eta_{n,j}^h.
\]

Observe that the same term $e^{\Delta TA}h$ appears in expressions of the derivatives $D\F_n(u).h$ and $D\G_n(u).h$. This term thus does not appear in $D\RR_n(u)$. This is a crucial property (which is not satisfied when the coarse integrator is the linear implicit Euler scheme). Observe also that $\RR_n=0$ if $F=0$, since the linear part is solved exactly.

\begin{proof}[Proof of Lemma~\ref{lm:LipR}--$({\rm i})$]
Since $F$ is globally Lipschitz continuous, applying the discrete Gronwall Lemma yields the following (almost sure) inequality:
\[
\underset{0\le j\le J}{\sup}|\eta_{n,j}^h|\le C|h|.
\]
Owing to Proposition~\ref{propo:A1} and using the global Lipschitz continuity property of $F$, it is straightforward to check that
\begin{align*}
|D\RR_n(u).h|_\alpha=&\big|D\F_n(u).h-D\G_n(u).h\big|_\alpha\\
\le& |\Delta Te^{\Delta TA}DF(u).h|_\alpha+\delta t\sum_{j=0}^{J-1}\big|e^{(J-j)\delta tA}DF(v_{n,j}).\eta_{n,j}^h\big|_\alpha\\
\le&\Delta T\|(-A)^\alpha e^{\Delta TA}\|_{\mathcal{L}(H)}|h|+\delta t\sum_{j=0}^{J-1}\|(-A)^\alpha e^{(J-j)\delta tA}\|_{\mathcal{L}(H)}|\eta_{n,j}^h|\\
\le& C_\alpha\Delta T^{1-\alpha}|h|,
\end{align*}
using the inequality $\delta t\sum_{j=0}^{J-1}\frac{1}{(J-j)\delta t)^\alpha}\le C_\alpha \Delta T^{1-\alpha}$. This concludes the proof.
\end{proof}

Observe that in the proof above, only the Lipschitz continuity property for $F$ was employed. In the proof below, the estimates stated in Assumption~\ref{ass:F} are crucial.

\begin{proof}[Proof of Lemma~\ref{lm:LipR}--$({\rm ii})$]
Using the notation above, and expressions of $D\F_n(u).h$ and of $D\G_n(u).h$, the derivative $D\RR_n(u).h=D\F_n(u).h-D\G_n(u).h$ is decomposed as follows: 
\begin{align*}
D\RR_n(u).h&=\Bigl(\sum_{j=0}^{J-1}\delta te^{(J-j)\delta tA}-\Delta Te^{\Delta TA}\Bigr)DF(u).h+\delta t\sum_{j=0}^{J-1}e^{(J-j)\delta tA}DF(u).\bigl(\eta_{n,j}^h-h\bigr)\\
&~+\delta t\sum_{j=0}^{J-1}e^{(J-j)\delta tA}\bigl(DF(v_{n,j})-DF(u)\bigr).\eta_{n,j}^h.
\end{align*}
The first term is treated as follows: using Assumption~\ref{ass:F},
\begin{align*}
\Big|(-A)^{-\alpha}&\Bigl(\sum_{j=0}^{J-1}\delta te^{(J-j)\delta tA}-\Delta Te^{\Delta TA}\Bigr)DF(u).h\Big|\\
&\le C\delta t\sum_{j=0}^{J-1}\|(-A)^{-\alpha-\beta+\kappa}\bigl(e^{J\delta tA}-e^{(J-j)\delta tA}\bigr)\|_{\mathcal{L}(H)}|(-A)^{\beta-\kappa}DF(u).h|\\
&\le C_{\alpha,\beta,\kappa} \delta t\sum_{j=0}^{J-1}(j\delta t)^{\alpha+\beta-\kappa}(1+|u|_{\beta})|h|_\beta\\
&\le C_{\alpha,\beta,\kappa}\Delta T^{1+\alpha+\beta-\kappa}(1+|u|_{\alpha})|h|_\beta,
\end{align*}
since $j\delta t\le J\delta t=\Delta T$, and $\alpha+\beta\le 2\alpha\le 1$.

The second term is treated as follows: observe that $\eta_{n,0}^h=h$, thus, using Assumption~\ref{ass:F},
\begin{align*}
\Big|(-A)^{-\alpha+\kappa}&\delta t\sum_{j=0}^{J-1}e^{(J-j)\delta tA}DF(u).\bigl(\eta_{n,j}^h-h\bigr)\Big|\le C_{\alpha,\kappa}\delta t\sum_{j=0}^{J-1}(1+|u|_{\alpha})\big|(-A)^{-\alpha+\kappa}\bigl(\eta_{n,j}^h-\eta_{n,0}^h\bigr)\big|.
\end{align*}
Note that the increment $\eta_{n,j}^h-\eta_{n,0}^h$ satisfies
\[
\eta_{n,j}^h-\eta_{n,0}^h=\bigl(e^{j\delta tA}-I\bigr)h+\delta t\sum_{\ell=0}^{j-1}e^{(j-\ell)\delta tA}D(v_{n,\ell}).\eta_{n,\ell}^h,
\]
with the inequality
\[
\big|(-A)^{-\alpha+\kappa}\bigl(e^{j\delta tA}-I\bigr)h\big|\le C_{\alpha,\beta,\kappa} (j\delta t)^{\alpha+\beta-\kappa}|h|_\beta.
\]
Using the inequality $|\eta_{n,j}^h|\le C|h|$, and using that $\kappa$ is chosen such that $\alpha+\beta+\kappa<1$, one obtains
\[
\Big|(-A)^{-\alpha}\delta t\sum_{j=0}^{J-1}e^{(J-j)\delta tA}DF(u).\bigl(\eta_{n,j}^h-h\bigr)\Big|\le C_{\alpha,\beta,\kappa}\Delta T^{1+\alpha+\beta-\kappa}(1+|u|_{\alpha})|h|_\beta.
\]
It remains to treat the third term: using $\beta+\kappa\le \alpha$, the last inequality from Assumption~\ref{ass:F} yields
\begin{align*}
\Big|(-A)^{-\alpha}&\delta t\sum_{j=0}^{J-1}e^{(J-j)\delta tA}\bigl(DF(v_{n,j})-DF(u)\bigr).\eta_{n,j}^h\Big|\\
&\le C_{\alpha,\beta,\kappa}\delta t\sum_{j=0}^{J-1}\Big|(-A)^{-\beta-\kappa}\bigl(DF(v_{n,j})-DF(v_{n,0})\bigr).\eta_{n,j}^h\Big|\\
&\le C_{\alpha,\beta,\kappa}\delta t\sum_{j=0}^{J-1}\bigl(1+|v_{n,0}|_\beta+|v_{n,j}|_\beta\bigr)|\eta_{n,j}^h|_\beta|v_{n,j}-v_{n,0}|_{-\beta-\kappa},
\end{align*}
with $|\eta_{n,j}^h|_\beta\le C|h|_\beta$. Note that the following moment bound holds true: there exists $C_{q,\beta}\in(0,\infty)$ such that $\vvvert v_{n,j}\vvvert_{q,\beta}\le C_{q,\beta}$. Moreover, the proof of the following regularity result is straightforward:
\[
\vvvert (-A)^{-\beta-\kappa}(v_{n,j}-v_{n,0})\vvvert_q\le C_{q,\alpha,\beta,\kappa}(j\delta t)^{\min\bigl(\alpha+\beta+\kappa,\frac12\bigr)}(1+|u|_\alpha),
\]
see Proposition~\ref{propo:bound_u} for a similar estimate when $\beta=0$, for the process $\bigl(u(t)\bigr)_{t\ge 0}$. Then applying H\"older inequality yields
\[
\vvvert (-A)^{-\alpha}\delta t\sum_{j=0}^{J-1}e^{(J-j)\delta tA}\bigl(DF(v_{n,j})-DF(u)\bigr).\eta_{n,j}^h\vvvert_q\le C_{q,\alpha,\beta}\Delta T^{1+\min\bigl(\alpha+\beta+\kappa,\frac12\bigr)}|h|_\beta(1+|u|_{\alpha}^2).
\]
Gathering the estimates then concludes the proof.
\end{proof}

\bibliography{parareal}
\bibliographystyle{plain}

\end{document}